\documentclass[10pt,reqno]{amsart}
\usepackage{cite}
\usepackage{amsmath,amscd,amsfonts,amsthm,amssymb,latexsym,mathrsfs}
\usepackage{tikz}
\usepackage{graphicx}
\usepackage{epic}
\usepackage{eepic}
\usepackage{url,color}

\theoremstyle{plain}
   \newtheorem{theorem}{Theorem}[section]
   \newtheorem{proposition}[theorem]{Proposition}
   \newtheorem{lemma}[theorem]{Lemma}
   \newtheorem{corollary}[theorem]{Corollary}
   \newtheorem{problem}{Problem}
   \newtheorem{conjecture}[theorem]{Conjecture}
   \newtheorem{question}{Question}
\theoremstyle{definition}
   
   \newtheorem{example}{Example}[section]

\theoremstyle{remark}
   \newtheorem{remark}[theorem]{Remark}

\numberwithin{equation}{section}

\author{Petter Br\"and\'en}\thanks{The author is a Wallenberg Academy fellow supported by a grant from the Knut and Alice Wallenberg Foundation. The author is also supported by a grant from the G\"oran Gustafsson Foundation.}

\address{Department of Mathematics, Royal Institute of Technology, SE-100 44 Stockholm,
Sweden}
\email{pbranden@kth.se}

\title[ Unimodality, Log-concavity and real--rootedness ]{ Unimodality, Log-concavity, real--rootedness \\ and beyond}

\begin{document}
\maketitle
\begin{center}
To appear in \emph{Handbook of Enumerative Combinatorics}, published by CRC Press
\end{center}
\bigskip

\tableofcontents

\thispagestyle{empty}

\newpage


\section{Introduction}
Many important  sequences in combinatorics are known to be log--concave or unimodal, but many are only conjectured to be so although several techniques using methods from combinatorics, algebra, geometry and analysis are now available. Stanley \cite{pbStanS} and Brenti \cite{pbBreS} have written extensive  surveys of various techniques that can be used to prove real--rootedness, log--concavity or unimodality. After a brief introduction and a short section on probabilistic consequences of real--rootedness, we will complement \cite{pbBreS, pbStanS} with a survey over new techniques that have been developed, and problems and conjectures that have been solved. I stress that this is not a comprehensive account of \emph{all} work that has been done in the area since \emph{op.~cit.}. The selection is certainly colored by my taste and knowledge. 

If $\mathcal{A}= \{a_k \}_{k=0}^n$ is a finite sequence of real numbers, then 
\begin{itemize}
\item $\mathcal{A}$ is \emph{unimodal} if there is an index $0\leq j\leq n$ such that
$$
a_0\leq \cdots \leq a_{j-1} \leq a_{j}\geq a_{j+1} \geq \cdots \geq a_n.
$$
\item $\mathcal{A}$ is \emph{log--concave} if 
$$
a_j^2 \geq a_{j-1}a_{j+1}, \ \ \ \ \mbox{ for all } 1 \leq j <n.
$$
\item the generating polynomial, $p_\mathcal{A}(x):= a_0+a_1x+\cdots + a_nx^n$, 
is called \emph{real--rooted} if all its zeros are real. By convention we also consider constant polynomials to be real--rooted. 
\end{itemize}
We say that the polynomial $p_\mathcal{A}(x)= \sum_{k=0}^n a_k x^k$ has a certain property if $\mathcal{A}=\{a_k\}_{k=0}^n$ does. 
The most fundamental sequence satisfying all of the properties above is the $n$th row of Pascal's triangle $\{\binom n k \}_{k=0}^n$. Log--concavity follows easily from the explicit formula $\binom n k = n!/k!(n-k)!$:
$$
\frac {{\binom n k}^2}{ {\binom n {k-1}} {\binom n {k+1}}} = \frac {(k+1)(n-k+1)}{k(n-k)} > 1. 
$$
The following lemma relates the three properties above. 
\begin{lemma}\label{pbelement}
Let $\mathcal{A} =\{a_k \}_{k=0}^n$ be a finite sequence of nonnegative numbers. 
\begin{itemize}
\item If $p_\mathcal{A}(x)$ is real--rooted, then the sequence $\mathcal{A}':=\{a_k/\binom n k  \}_{k=0}^n$ is log--concave. 
\item If $\mathcal{A}'$ is log-concave, then so is $\mathcal{A}$. 
\item If $\mathcal{A}$ is log-concave and positive, then $\mathcal{A}$ is unimodal. 
\end{itemize}
\end{lemma}
\begin{proof}
Suppose $p_\mathcal{A}(x)$ is real--rooted. 
Let $a_k = \binom n k b_k$, for $1\leq k \leq n$. By the Gauss--Lucas theorem below, the polynomial 
\begin{equation}\label{pb1nder}
\frac 1 n p_{\mathcal{A}}'(x) = \sum_{k=0}^n \frac k n \binom n k b_k x^{k-1} =  \sum_{k=0}^{n-1} \binom {n-1} {k} b_{k+1} x^{k}
\end{equation}
is real--rooted. The operation 
\begin{equation}\label{pb2nder}
x^n p_{\mathcal{A}}(1/x) = \sum_{k=0}^n \binom n k b_{n-k} x^{k},
\end{equation}
preserves real--rootedness. Let $1 \leq j \leq n-1$. Applying the operations 
\eqref{pb1nder} and \eqref{pb2nder} appropriately to $p_\mathcal{A}(x)$, we end up with the real--rooted polynomial
$$
b_{j-1} + 2 b_j x + b_{j+1} x^2,
$$
and thus $b_j^2 \geq b_{j-1}b_{j+1}$. This proves the first statement. 

The term-wise (Hadamard) product of a positive and log--concave sequence and a log--concave sequence is again log--concave. Since $\{\binom n k \}_{k=0}^n$ is positive and log--concave, the second statement follows. 

The third statement follows directly from the definitions. 
 \end{proof}

\begin{example}
Natural examples of log--concave polynomials which are not  real--rooted are the $q$-\emph{factorial} polynomials, 
$$
[n]_q! = [n]_q\cdot [n-1]_q \cdots [2]_q\cdot [1]_q,
$$
where $[k]_q= 1+q+\cdots+q^{n-1}$. The polynomial $[n]_q!$ is the generating polynomial for the \emph{number of inversions} over the symmetric group $\mathfrak{S}_n$: 
$$
[n]_q! = \sum_{\pi \in \mathfrak{S}_n} q^{{\rm inv}(\pi)},
$$
where 
$$
{\rm inv}(\pi)= |\{ 1 \leq i<j \leq n : \pi(i) > \pi(j)\}|, 
$$
see \cite{pbStanEn}. The easiest way to see that  $[n]_q!$ is log--concave is to observe that $[k]_q$ is log--concave. Log--concavity of $[n]_q!$ then follows from the fact that if $A(x)$ and $B(x)$ are generating polynomials of positive log--concave sequences, then so is $A(x)B(x)$, see \cite{pbStanS}. 
\end{example}

\begin{example}
Examples of unimodal sequences that are not log--concave are the $q$--\emph{binomial coefficients} 
$$
{n \brack k}_q= \frac {[n]_q!} {[k]_q![n-k]_q!}. 
$$
These are polynomials with nonnegative coefficients 
\begin{equation}\label{pbqbin}
{n \brack k}_q= a_0(n,k)+ a_1(n,k)q + \cdots + a_{k(n-k)}(n,k) q^{k(n-k)}, 
\end{equation}
which are unimodal and symmetric. There are several proofs of this fact, see \cite{pbStanS}. For example the Cayley--Sylvester theorem, first stated by Cayley in the 1850's and proved by Sylvester in 1878, implies unimodality of \eqref{pbqbin}, see \cite{pbStanS}. However ${4 \brack 2}_q= 1 +q +2q^2+q^3+q^4$, which is not log--concave. 
\end{example}
For a proof of the following fundamental theorem we refer to \cite{pbRS}. 
\begin{theorem}[The Gauss--Lucas theorem]
Let $f(x) \in \mathbb{C}[x]$ be a polynomial of degree at least one. All zeros of 
$f'(x)$ lie in the convex hull of the zeros of $f(x)$. 
\end{theorem}
%

\begin{example}
Let $\{S(n,k)\}_{k=0}^n$ be the \emph{Stirling numbers of the second kind}, see \cite{pbStanEn}. Then $\bar{S}(n,k):=k!S(n,k)$ counts the number of surjections from $[n] := \{1,2,\ldots,n\}$ to $[k]$. For a surjection $f : [n+1] \rightarrow [k]$, let $j=f(n+1)$. Conditioning on whether  $|f^{-1}(\{j\})|=1$ or $|f^{-1}(\{j\})|>1$, one sees that 
\begin{equation}\label{pbbars}
\bar{S}(n+1,k) = k\bar{S}(n,k-1)+k\bar{S}(n,k), \ \ \ \mbox{ for all } 1 \leq k \leq n+1. 
\end{equation}
Let $E_n(x) = \sum_{k=1}^n \bar{S}(n,k) x^k$. Then \eqref{pbbars} translates as 
$$
E_{n+1}(x)= xE_n(x)+ x(x+1)E_n'(x)= x\frac d {dx} \Big((x+1)E_n(x)\Big).
$$
By induction and the Gauss--Lucas theorem, we see that $E_n(x)$ is real--rooted,  and that all its zeros lie in the interval $[-1,0]$ for all $n \geq 1$. Later, in Example~\ref{pbBesselex} we will see that the operation of dividing the $k$th coefficient by $k!$, for each $k$, preserves real--rootedness. Hence also the polynomials 
$
\sum_{k=1}^n S(n,k)x^k
$, $n \geq 1$, 
 are real--rooted. 
\end{example}

A generalization of finite nonnegative sequences with real--rooted generating polynomials is that of P\'olya frequency sequences. A sequence $\{a_k\}_{k=0}^\infty \subseteq \mathbb{R}$ is a \emph{P\'olya frequency sequence} (PF for short) if all minors of the infinite Toeplitz matrix  $(a_{i-j})_{i,j=0}^\infty$ are nonnegative. In particular, PF sequences are log--concave. PF sequences are characterized by the following theorem of Edrei \cite{pbEdrei}, first conjectured by Schoenberg. 
\begin{theorem}\label{pbEdthm}
A sequence  $\{a_k\}_{k=0}^\infty \subseteq \mathbb{R}$  of real numbers is {\rm PF} if and only its generating function may be expressed as 
$$
\sum_{k=0}^\infty a_kx^k= Cx^m e^{ax} { {\prod_{k=0}^\infty (1+ \alpha_k x)} } \Big/{ {\prod_{k=0}^\infty (1- \beta_k x) }},
$$
where $C,a \geq 0$, $m \in \mathbb{N}$, $\alpha_k, \beta_k \geq 0$ for all $k \in \mathbb{N}$, and $\sum_{k=0}^\infty (\alpha_k + \beta_k) < \infty$.  
\end{theorem}
Hence a finite nonnegative sequence is PF if and only its generating polynomial is real--rooted. This was first proved by Aissen, Schoenberg and Whitney \cite{pbASW}. Theorem \ref{pbEdthm} provides --- at least in theory --- a method of proving combinatorially that a combinatorial polynomial with nonnegative coefficients is real--rooted. Namely to find a combinatorial interpretation of the minors of $(a_{i-j})_{i,j=0}^\infty$. This method was used by e.g. Gasharov \cite{pbGash} to prove that the independence polynomial of a $(3+1)$-free graph is real--rooted. For more on PF sequences in combinatorics, see \cite{pbBreTh}. 

\section{Probabilistic consequences of real--rootedness}
Below we will explain two  useful probabilistic consequences of real--rootedness. For further consequences, see Pitman's survey \cite{pbPitman}. 
If $X$ is a random variable taking values in $\{0,\ldots, n\}$, let $a_k=\mathbb{P}[X=k]$ for $0\leq k \leq n$, and let
$$
p_X(t)= a_0 + a_1t+\cdots + a_nt^n, 
$$
be the \emph{partition function} of $X$. Then $X$ has mean 
$$
\mu = \mathbb{E}[X]= \sum_{k=0}^n k\mathbb{P}[X=k]= p'_X(1), 
$$
and variance 
$$
{\rm Var}(X)=  \mathbb{E}[X^2]-\mu^2= p_X''(1)+p'_X(1)-p'_X(1)^2. 
$$
The following theorem of Bender \cite{pbBend} has been used on numerous occasions to prove asymptotic normality of combinatorial sequences, see e.g. \cite{pbBend,pbBen,pbBonaGS}. 
\begin{theorem}\label{pbnormal}
Let $\{X_n\}_{n=1}^\infty$ be a sequence of random variables taking values in  $\{0,1,\ldots, n\}$ such that 
\begin{enumerate}
\item $p_{X_n}(t)$ is real--rooted for all $n$, and 
\item ${\rm Var}(X_n) \rightarrow \infty$.
\end{enumerate}
Then the distribution of the random variable 
$$
\frac {X_n - \mathbb{E}[X_n]} { \sqrt{{\rm Var}(X_n)} }
$$
converges to the standard normal distribution $N(0,1)$ as $n \to \infty$. 
\end{theorem}

\begin{example}\label{pbStirling}
Let $X_n$ be the random variable on the symmetric group $\mathfrak{S}_n$ counting the number of cycles in a uniform random permutation. Since the number of permutations in $\mathfrak{S}_n$ with exactly $k$ cycles is the \emph{signless Stirling number} of the first kind $c(n,k)$ (see \cite{pbStanEn}), 
$$
p_{X_n}(t) = \frac 1 {n!} x(x+1) \cdots (x+n-1).
$$
Thus $X_n$ has mean $H_{n}= 1+1/2+\cdots + 1/n$ and variance 
$$
\sigma_n^2 = H_{n}- \sum_{k=1}^{n} k^{-2}.
$$
Hence the distribution of the random variable 
$$
\frac {X_n - H_{n}} { \sigma_n }
$$
converges to the standard normal distribution $N(0,1)$ as $n \to \infty$. 
\end{example}
For more examples using Theorem \ref{pbnormal}, see \cite{pbBend}, and for recent examples, see \cite{pbBen,pbBonaGS}. 

A simple consequence of Lemma \ref{pbelement} is that if a polynomial $a_0 + a_1x+ \cdots + a_nx^n$ has only real and nonpositive zeros, then there is either a unique index $m$ such that $a_m = \max_k a_k$, or two consecutive indices $m\pm 1/2$ (whence $m$ is a half-integer) such that $a_{m\pm 1/2} = \max_k a_k$. The number $m=m(\{a_k\}_{k=0}^n)$ is called the mode of $\{a_k\}_{k=0}^n$. A theorem of Darroch \cite{pbDarr} enables us to easily compute the mode. 

\begin{theorem}\label{pbdarr}
Suppose $\{a_k\}_{k=0}^n$ is a sequence of nonnegative numbers such that the polynomial $p(x)= a_0 + a_1x+ \cdots + a_nx^n$ is real--rooted. If $m$ is the mode of $\{a_k\}_{k=0}^n$, and $\mu:= p'(1)/p(1)$ its mean, then 
$$
\lfloor \mu \rfloor \leq m \leq \lceil \mu \rceil . 
$$
\end{theorem}
Applying Theorem \ref{pbdarr} to the signless Stirling numbers of the first kind $\{c(n,k)\}_{k=1}^n$ (Example \ref{pbStirling}),  we see that 
$$
  \lfloor H_{n} \rfloor \leq m(\{c(n,k)\}_{k=1}^n) \leq \lceil H_{n} \rceil. 
$$

\section{Unimodality and $\gamma$-nonnegativity}
We say that the sequence $\{h_k\}_{k=0}^d$ is \emph{symmetric} with \emph{center of symmetry} $d/2$ if $h_k=h_{d-k}$ for all $0 \leq k \leq  d$. A property called $\gamma$--\emph{nonnegativity}, which implies symmetry and unimodality,  has recently been considered in topological, algebraic and enumerative combinatorics. 

The linear space of polynomials $h(x)= \sum_{k=0}^d h_k x^k \in \mathbb{R}[x]$ which are symmetric with center of symmetry $d/2$ has a  basis 
$$B_d:=\{x^k(1+x)^{d-2k}\}_{k=0}^{\lfloor d/2 \rfloor}.$$
If $h(x)= \sum_{k=0}^{\lfloor d/2 \rfloor}\gamma_k x^k(1+x)^{d-2k}$, we call $\{\gamma_k\}_{k=0}^{\lfloor d/2 \rfloor}$ the $\gamma$-\emph{vector} of $h$. Since the binomial numbers are unimodal, having a nonnegative $\gamma$-vector implies unimodality of $\{h_k\}_{k=0}^n$. If the $\gamma$-vector of $h$ is nonnegative, then we say that $h$ is $\gamma$-\emph{nonnegative}. Let $\Gamma_+^d$ be the convex cone  of polynomials that have nonnegative coefficients when expanded in $B_d$. Clearly 
\begin{equation}\label{pbgammaprod}
\Gamma_+^m \cdot \Gamma_+^n := \{ fg : f \in \Gamma_+^m \mbox{ and } g \in \Gamma_+^n\} \subseteq \Gamma_+^{m+n}.
\end{equation}
\begin{remark}\label{pbgammaroots}
Suppose $h(x)= \sum_{k=0}^d h_k x^k \in \mathbb{R}[x]$ is the  generating polynomial of a nonnegative and symmetric sequence with center of symmetry $d/2$. If all its zeros are real, then we may pair the negative zeros into reciprocal pairs  
$$
h(x)= Ax^k \prod_{i=1}^\ell (x+\theta_i)(x+1/\theta_i)= Ax^k \prod_{i=1}^\ell ((1+x)^2+(\theta_i+1/\theta_i-2)x), 
$$
where $A>0$. Since  $x$ and $(1+x)^2+(\theta_i+1/\theta_i-2)x$ are polynomials in $\Gamma_+^1$, we see that $h$ is $\gamma$-nonnegative by \eqref{pbgammaprod}.
\end{remark}

\subsection{An action on permutations}\label{pbSecAct}
There is a natural $\mathbb{Z}_2^n$-action on $\mathfrak{S}_n$, first considered in a modified version by Foata and Strehl \cite{pbFoSt}, which has been used to prove $\gamma$-nonnegativity. Let $\pi = a_1a_2\cdots a_n \in \mathfrak{S}_n$ be a permutation written as a word ($\pi(i)=a_i$),  
and set $a_0=a_{n+1}=n+1$. 
If $k \in [n]$, then $a_k$ is a
\begin{itemize}
\item[] {\em valley} if $a_{k-1}> a_k < a_{k+1}$,
\item[] {\em peak} if $a_{k-1}<  a_k >a_{k+1} $,
\item[] {\em double ascent} if  $a_{k-1} <a_k < a_{k+1}$, and 
\item[] {\em double descent} if $a_{k-1}> a_k >a_{k+1}$.
\end{itemize}
Define functions $\varphi_x : \mathfrak{S}_n \rightarrow \mathfrak{S}_n$, $x \in [n]$, as follows:
\begin{itemize}
\item If $x$ is a double descent, then $\varphi_x(\pi)$ is 
obtained by moving  $x$  into the slot between the first pair of letters 
$a_i, a_{i+1}$ to the right of $x$ such that $a_i < x < a_{i+1}$;
 
\item If $x$ is a double ascent, then $\varphi_x(\pi)$ is 
obtained by moving $x$ to the slot between the first pair of letters 
$a_i, a_{i+1}$ to the left of $x$ such that $a_i > x > a_{i+1}$;
\item If $x$ is a valley or a peak, then  $\varphi_x(\pi)=\pi$. 
\end{itemize}

There is a geometric interpretation of the functions $\varphi_x$, $x \in [n]$, first considered in \cite{pbSWG}. 
 Let $\pi=a_1a_2\cdots a_n \in \mathfrak{S}_n$ and imagine 
marbles at the points $(i,a_i) \in \mathbb{N} \times \mathbb{N}$, for  $i=0,1,\ldots,n+1$. For $i = 0,1,\ldots, n$ 
connect $(i,a_i)$ and $(i+1,a_{i+1})$ with a wire. Suppose 
gravity acts on the marbles, and that $x$ is not at an equilibrium. If $x$ is released it will slide and stop when 
it has reached the same height again. The resulting permutation is 
$\varphi_x(\pi)$, see Fig.~\ref{pbgraphi}.

\begin{figure}
\begin{center}
 \includegraphics[height=5cm]{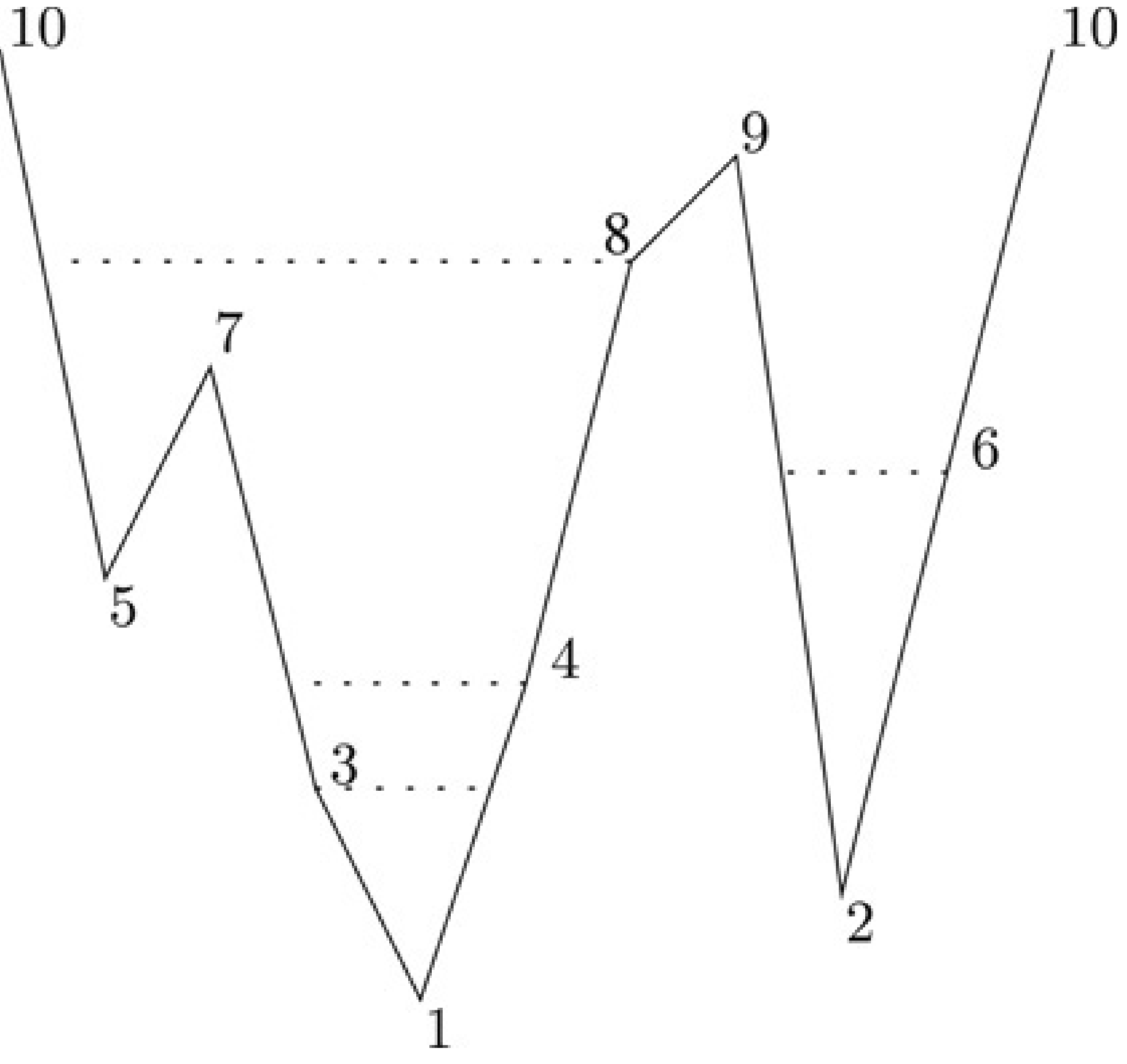}
\end{center}
\caption{Graphical representation of $\pi = 573148926$. The dotted lines indicate where 
the double ascents/descents move to.}\label{pbgraphi}
\end{figure}

The functions $\varphi_x$ are commuting involutions. Hence  
for any subset $S \subseteq [n]$, we may define the 
function $\varphi_S : \mathfrak{S}_n \rightarrow \mathfrak{S}_n$ by
$$
\varphi_S(\pi) = \prod_{x \in S} \varphi_x(\pi).
$$
Hence the group $\mathbb{Z}_2^n$ acts on $\mathfrak{S}_n$ via the 
functions $\varphi_S$, $S \subseteq [n]$. For example 
$$\varphi_{\{2,3,7,8\}}(573148926)=857134926.
$$

For $\pi \in \mathfrak{S}_n$, let ${\rm Orb}(\pi)= \{ g(\pi) : g \in \mathbb{Z}_2^n\}$ be the 
orbit of $\pi$ under the action. There is a unique element in ${\rm Orb}(\pi)$ 
which has no double descents and which we denote by  $\hat{\pi}$. 
\begin{theorem}\label{orb}
Let $\pi= a_1a_2\cdots a_n \in \mathfrak{S}_n$. Then 
$$
\sum_{\sigma \in {\rm Orb}(\pi)} x^{{\rm des} (\sigma) } = 
x^{{\rm des} (\hat{\pi})} (1+x)^{n-1-2{\rm des} (\hat{\pi})}=x^{{\rm peak} (\pi)} (1+x)^{n-1-2{\rm peak}(\pi)},
$$
where ${\rm des}(\pi) = |\{ i \in [n] : a_i > a_{i+1}\}|$ and ${\rm peak}(\pi)=|\{ i \in [n] : a_{i-1}<a_i>a_{i+1}\}|$.
\end{theorem}
\begin{proof}
If $x$ is a double ascent in $\pi$ then 
${\rm des} (\varphi_x(\pi))= {\rm des} (\pi) +1$. It follows that 
$$
\sum_{\sigma \in {\rm Orb}(\pi)} x^{{\rm des} (\sigma) }= x^{{\rm des}(\hat{\pi})}(1+x)^a,
$$
where $a$ is the number of double ascents in $\hat{\pi}$. If we 
delete all double ascents from $\hat{\pi}$ we get 
an alternating permutation 
$$
n+1 > b_1 < b_2 > b_3 <\cdots > b_{n-a}< n+1,
$$
with the same number of descents. Hence  
$n-a = 2{\rm des} (\hat{\pi}) +1$. Clearly ${\rm des}(\hat{\pi})={\rm peak}(\pi)$ and the theorem follows. 
 \end{proof}  
For a subset $T$ of $\mathfrak{S}_n$ let 
$$
A(T;x):=\sum_{\pi \in T} x^{{\rm des}(\pi)}.
$$
\begin{corollary}\label{pbcorre} 
If $T \subseteq \mathfrak{S}_n$ is invariant under the $\mathbb{Z}_2^n$-action, then 
$$
A(T;x) = \sum_{i=0}^{\lfloor n/2 \rfloor} \gamma_i(T)x^i(1+x)^{n-1-2i}, 
$$
where 
$$
\gamma_i(T)= 2^{-n+1+2i}|\{ \pi \in T : {\rm peak}(\pi)=i\}|. 
$$
In particular $A(T,x)$ is $\gamma$-nonnegative. 
\end{corollary}
\begin{proof}
It is enough to prove the theorem for an orbit of a permutation $\pi \in \mathfrak{S}_n$. Since the number of  peaks is constant on ${\rm Orb}(\pi)$ the equality follows from Theorem~\ref{orb}.  \end{proof}

\begin{example}
Recall that the \emph{Eulerian polynomials} are defined by 
\begin{equation}\label{pbEul}
A_n(x) = \sum_{\pi \in \mathfrak{S}_n}x^{{\rm des}(\pi)+1},  
\end{equation}
see \cite{pbStanEn}. By Corollary \ref{pbcorre}, 
$$
 A_n(x)/x= \sum_{i=0}^{\lfloor n/2 \rfloor} \gamma_{ni} x^i(1+x)^{n-1-2i},
$$
where 
$$
\gamma_{ni}= 2^{-n+1+2i}|\{ \pi \in \mathfrak{S}_n : {\rm peak}(\pi)=i\}|. 
$$
\end{example}

\begin{example}
This example is taken from \cite{pbAct}. The \emph{stack-sorting operator} $S$ may be defined recursively on permutations of finite subsets of 
$\{1,2,\ldots \}$ as follows. If 
$w$ is empty, then $S(w):=w$. If  
$w$ is nonempty,  
write $w$ as the concatenation $w=LmR$ where $m$ is the greatest element 
of $w$, and $L$ and $R$ are the subwords to the left and right of $m$, 
respectively. Then $S(w):=S(L)S(R)m$. 

If $\sigma, \tau \in \mathfrak{S}_n$ are in the same orbit under 
the $\mathbb{Z}_2^n$-action, then it is not hard to prove that $S(\sigma)=S(\tau)$, see \cite{pbAct}.  Let $r \in \mathbb{N}$. 
A permutation $\pi \in \mathfrak{S}_n$ is said to be \emph{$r$-stack sortable} if  
$S^r(\pi)=12\cdots n$. Denote by $\mathfrak{S}_n^r$ the set of $r$-stack sortable permutations in $\mathfrak{S}_n$. Hence $\mathfrak{S}_n^r$ is invariant under the $\mathbb{Z}_2^n$-action for all $n,r \in \mathbb{N}$, so Corollary~\ref{pbcorre} applies to prove that for all $n,r \in \mathbb{N}$  
$$
A(\mathfrak{S}_n^r;x) = \sum_{i=0}^{\lfloor n/2 \rfloor} \gamma_i(\mathfrak{S}_n^r)x^i(1+x)^{n-1-2i}, 
$$
where 
$$
\gamma_i(\mathfrak{S}_n^r)= 2^{-n+1+2i} |\{ \pi \in \mathfrak{S}_n^r : {\rm peak}(\pi)=i\}|.
$$
Unimodality and symmetry of $A(\mathfrak{S}_n^r;x)$ was first proved by Bona \cite{pbBon1}. Bona conjectured that $A(\mathfrak{S}_n^r;x)$ is real--rooted for all $n,r \in \mathbb{N}$. This conjecture remains open for all $3\leq r \leq n-3$, see \cite{pbAct}. 

More generally, if $A \subseteq \mathfrak{S}_n$, then the polynomial 
$$
\sum_{{\pi \in \mathfrak{S}_n} \atop S(\pi) \in A} x^{{\rm des}(\pi)}
$$
is $\gamma$--nonnegative. 
\end{example}

Postnikov, Reiner and Williams \cite{pbPRW} modified the $\mathbb{Z}_2^n$-action to prove Gal's conjecture (see Conjecture \ref{pbGalC}) for so called chordal nestohedra. 

In \cite{pbShaWa}, Shareshian and Wachs proved refinements of the $\gamma$-positivity of Eulerian polynomials. Let 
$$
A_{n}(q,p,s,t)= \sum_{k=0}^n A_{n,k}(q,p,t) s^k = \sum_{\sigma \in \mathfrak{S}_n} q^{{\rm maj}(\sigma)}p^{{\rm des}(\sigma)}t^{{\rm exc}(\sigma)} s^{{\rm fix}(\sigma)}, 
$$
where 
\begin{align*}
{\rm exc}(\sigma) &= |\{i : \sigma(i)>i\}|,\\
{\rm fix}(\sigma) &= |\{i : \sigma(i)=i\}|,  \mbox{ and } \\
{\rm maj}(\sigma) &= \sum_{i : \sigma(i) > \sigma(i+1)} i. 
\end{align*}
\begin{theorem}
Let $B_d= \{t^k(1+t)^{d-2k}\}_{k=0}^{\lfloor d/2 \rfloor}$. 
\begin{enumerate}
\item The polynomial $A_{n,0}(q,p,q^{-1}t)$ has coefficients in $\mathbb{N}[q,p]$ when expanded in $B_n$. 
\item If $1\leq k \leq n$, then $A_{n,k}(q,1,q^{-1}t)$  has coefficients in $\mathbb{N}[q]$ when expanded in $B_{n-k}$.
\item The polynomial $A_n(q,1,1,q^{-1}t)$ has coefficients in $\mathbb{N}[q]$ when expanded in $B_{n-1}$.
\end{enumerate}
\end{theorem}

Gessel \cite{pbGesselPC} has conjectured a fascinating property which resembles $\gamma$-nonnegativity for the joint distribution of descents and inverse descents: 
\begin{conjecture}[Gessel, \cite{pbGesselPC,pbAct,pbPete}]
If $n$ is a positive integer, then there are nonnegative numbers $c_n(k,j)$ for all $k,j \in \mathbb{N}$ such that  
\begin{equation}\label{pbgest}
\sum_{\pi \in \mathfrak{S}_n} x^{{\rm des}(\pi)}y^{{\rm des}(\pi^{-1})} = \sum_{{k,j \in \mathbb{N}} \atop {k+2j\leq n-1}} c_n(k,j) (x+y)^k(xy)^j(1+xy)^{n-k-1-2j}. 
\end{equation}
\end{conjecture}
The existence of integers $c_n(k,j)$ satisfying \eqref{pbgest} follows from symmetry properties, see \cite{pbPete}. The open problem is nonnegativity.

\subsection{$\gamma$-nonnegativity of $h$-polynomials}\label{pbSecGh}
In topological combinatorics the $\gamma$-vectors were introduced in the context of  face numbers of simplicial complexes \cite{pbSign,pbGal}. The $f$-\emph{polynomial} of a $(d-1)$-dimensional simplicial complex $\Delta$ is 
$$
f_\Delta(x)= \sum_{k=0}^d f_{k-1}(\Delta) x^k,
$$
where $f_k(\Delta)$ is the number of $k$-dimensional faces in $\Delta$, and $f_{-1}(\Delta):=1$. The $h$-\emph{polynomial} is defined by 
\begin{align}\label{pbhpol}
h_{\Delta}(x)&= \sum_{k=0}^d h_k(\Delta) x^k= (1-x)^df_{\Delta}(x/(1-x)), \ \ \mbox{ or equivalently,} \\
f_{\Delta}(x)&= (1+x)^d h_{\Delta}(x/(1+x)). \nonumber 
\end{align}
Hence $f_\Delta(x)$ and $h_\Delta(x)$ contain the same information. 
If $\Delta$ is a $(d-1)$-dimensional homology sphere, then the \emph{Dehn--Sommerville relations} (see \cite{pbStanComCom}) tell us that $h_\Delta(x)$ is symmetric, so we may expand it in the basis $B_d$. Recall that a simplicial complex $\Delta$ is \emph{flag} if all minimal non-faces of $\Delta$ have cardinality two. Motivated by the Charney--Davis conjecture  below, Gal made the following intriguing conjecture: 
\begin{conjecture}[Gal, \cite{pbGal}]\label{pbGalC}
If $\Delta$ is a flag homology sphere, then $h_\Delta(x)$ is $\gamma$--nonnegative. 
\end{conjecture}
Gal's conjecture is true for dimensions less than five, see \cite{pbGal}. 
If $h_\Delta(x)$ is symmetric with center of symmetry $d/2$, then $h_\Delta(-1)= 0$ if $d$ is odd, and $h_\Delta(-1)=(-1)^{d/2}\gamma_{d/2}(\Delta)$ if $d$ is even. Hence Gal's conjecture implies the Charney--Davis conjecture:
\begin{conjecture}[Charney--Davis, \cite{pbChDa}]
If $\Delta$ is a flag $(d-1)$-dimensional homology sphere, where $d$ is even, then $(-1)^{d/2}h_\Delta(-1)$ is nonnegative. 
\end{conjecture}
Postnikov, Reiner and Williams \cite{pbPRW} proposed a natural extension of  Conjecture~\ref{pbGalC}. 
\begin{conjecture}\label{pbPRWC}
If $\Delta$ and $\Delta'$ are flag homology spheres such that $\Delta'$ geometrically subdivides $\Delta$, then the $\gamma$-vector of $\Delta'$ is entry-wise larger or equal to the  $\gamma$-vector of $\Delta$. 
\end{conjecture}
Conjecture \ref{pbPRWC} was proved for dimensions $\leq 4$ in a slightly stronger form by Athanasiadis \cite{pbAthPac}. In \cite{pbAthPac}, Athanasiadis also proposes an analog of Gal's conjecture for local $h$-polynomials.

\subsection{Barycentric subdivisions}
The collection of faces of a  regular cell complex $\Delta$ are naturally partially ordered by inclusion; if $F$ and $G$ are open cells in $\Delta$, then $F \leq G$ if $F$ is contained in the closure of $G$, where we assume that the empty face is contained in every other face. A \emph{Boolean cell complex} is a regular cell complex such that each interval $[\emptyset, F]= \{G \in \Delta : G \leq F\}$ is isomorphic to a Boolean lattice. Hence simplicial complexes are Boolean. The \emph{barycentric subdivision}, ${\rm sd}(\Delta)$, of a Boolean cell complex, $\Delta$, is the simplicial complex whose $(k-1)$-dimensional faces are  strictly increasing flags  
$$
F_1<F_2<\cdots <F_k,
$$
where $F_j$ is a nonempty face of $\Delta$ for each $1\leq j \leq k$.  The $f$-polynomials and $h$-polynomials for cell complexes are defined just as for simplicial complexes. 

Brenti and Welker \cite{pbBrWe} investigated positivity properties, such as real--rootedness and $\gamma$-positivity, of the $h$-polynomials of complexes under taking barycentric subdivisions. This was done by using analytic properties --- obtained in \cite{pbTrans, pbBrWe} --- of the linear operator that takes the $f$-polynomial of a Boolean complex to the $f$-polynomial of its barycentric subdivision. These analytic properties will be discussed in Section \ref{pbSecSO}. In this section we describe the topological consequences of the analytic properties. 

Let $\mathcal{E} : \mathbb{R}[x] \rightarrow \mathbb{R}[x]$ be the linear operator defined by its image on the binomial basis:
$$
\mathcal{E} \binom x k = x^k, \quad \mbox{ for all } k\in \mathbb{N}, \mbox{ where } \binom x k = \frac{x(x-1)\cdots (x-k+1)}{k!}.
$$
The operator $\mathcal{E}$ appears in several combinatorial settings. Using the binomial theorem one sees
$$
\mathcal{E}(f)(x) = \sum_{n=0}^\infty f(n) \frac {x^n}{(1+x)^{n+1}}.
$$
It follows e.g. from the theory of $P$--partitions (or from \eqref{pbbars} and induction) that 
$$
\mathcal{E}(x^n) = E_n(x)= \sum_{k=1}^n k!S(n,k)x^k, \ \ \mbox{ for all } n \geq 1,
$$
where $\{S(n,k)\}_{k=0}^n$ are the Stirling numbers of the second kind, see \cite{pbStanEn,pbWa2}.

The following lemma was proved by Brenti and Welker \cite{pbBrWe}.

\begin{lemma}\label{pbsdd}
For any Boolean cell complex $\Delta$, 
$$
f_{{\rm sd}(\Delta)}= \mathcal{E}(f_\Delta).
$$
\end{lemma}
\begin{proof}
By definition 
$$
f_{{\rm sd}(\Delta)}(x) = \sum_{F \in \Delta} W_F(x), 
$$
where $W_\emptyset =1$ and 
$$
W_F(x)= \sum_{k=1}^{\dim F+1} x^k |\{ \emptyset < F_1 < \cdots < F_{k} =F\}|,
$$
if $F \neq \emptyset$. Since $\Delta$ is Boolean, there is a one--to--one correspondence between flags $\emptyset < F_1 < \cdots < F_{k} =F$, $1 \leq k \leq \dim F +1$, and ordered set-partitions of $[n]$, where $n= \dim F+1$. Hence $W_F(x)= E_n(x)= \mathcal{E}(x^n)$, and the lemma follows. 
 \end{proof}

\begin{lemma}\label{pbsdsym}
Let $\Delta$ be a $(d-1)$-dimensional Boolean cell complex. If $h_\Delta(x)$ is symmetric, then so is $h_{{\rm sd}(\Delta)}(x)$. 
\end{lemma}

\begin{proof}
By \eqref{pbhpol}, $h_\Delta(x)$ is symmetric if and only if $(-1)^df_\Delta(-1-x)=f_\Delta(x)$. Let $I : \mathbb{R}[x] \rightarrow \mathbb{R}[x]$ be the algebra automorphism defined by $I(x)=-1-x$. It was observed in \cite[Lemma 4.3]{pbTrans}  that 
\begin{equation}\label{pbSE}
I \circ \mathcal{E}= \mathcal{E} \circ I,
\end{equation} 
from which the lemma follows. 
 \end{proof}

\begin{corollary}[\cite{pbBrWe}]\label{pbdebo}
Let $\Delta$ be a Boolean cell complex. If the $h$-polynomial of $\Delta$ has nonnegative coefficients, then all zeros of $h_{{\rm sd}(\Delta)}(x)$ are nonpositive and simple. 

If  $h_{\Delta}(x)$ is also symmetric, then  $h_{{\rm sd}(\Delta)}(x)$ is $\gamma$-nonnegative. 
\end{corollary}

\begin{proof}
The first conclusion follows immediately from  Theorem \ref{pbcons}, Lemma~\ref{pbsdd} and \eqref{pbhpol}. The second conclusion follows from Remark~\ref{pbgammaroots} and Lemma~\ref{pbsdsym}. 
 \end{proof}
The second conclusion of Corollary \ref{pbdebo} was strengthened in \cite{pbNPT}, where it was shown that with the same hypothesis, the $\gamma$-vector of ${\rm sd}(\Delta)$ is the $f$-vector of a balanced simplicial complex. 

If $\Delta$ is a Boolean cell complex and $k$ is a positive integer, let ${\rm sd}^k(\Delta)$ be the simplicial complex obtained by a $k$-fold application of the subdivision operator ${\rm sd}$. Most of the following corollary appears in \cite{pbBrWe}. 

\begin{corollary}\label{pblimcor}
Let $\Delta$ be a $(d-1)$-dimensional Boolean cell complex with reduced Euler characteristic $\tilde{\chi}(\Delta)$, where $d\geq 2$. There exists a number $N(\Delta)$ such that 
\begin{itemize}
\item[(1)] all zeros of $h_{{\rm sd}^n(\Delta)}(x)$ are real and simple for all $n \geq N(\Delta)$, 
\item[(2)] if $(-1)^{d-1}\tilde{\chi}(\Delta) \geq 0$, then all zeros of $h_{{\rm sd}^n(\Delta)}(x)$ are nonpositive and simple for all $n \geq N(\Delta)$,
\item[(3)] if $(-1)^{d-1}\tilde{\chi}(\Delta) < 0$, then all zeros of $h_{{\rm sd}^n(\Delta)}(x)$ except one are nonpositive and simple for all $n \geq N(\Delta)$. 
\end{itemize} 
Moreover 
\begin{equation}\label{pblimsdn}
\lim_{n \to \infty} \frac 1 {d!^n}f_{{\rm sd}^n(\Delta)}(x)= f_{d-1}(\Delta) p_d(x), 
\end{equation}
where $p_d(x)$ is the unique monic degree $d$ eigenpolynomial of $\mathcal{E}$ (see Theorem \ref{pbeigen}).
\end{corollary}
\begin{proof}
The identity \eqref{pblimsdn} follows from the proof of Theorem \ref{pbeigen} by choosing $f= f_{\Delta}(x)/ f_{d-1}(\Delta)$. By Theorem \ref{pbeigen}, all zeros of $p_d(x)$ are real, simple and lie in the interval  $[-1,0]$. In view of \eqref{pblimsdn} all zeros of $f_{{\rm sd}^n(\Delta)}(x)$ will be real and simple for $n$ sufficiently large. The same holds for $h_{{\rm sd}^n(\Delta)}(x)$ by \eqref{pbhpol}. 

Assume $(-1)^{d-1}\tilde{\chi}(\Delta) \geq 0$. By Theorem \ref{pbeigen}, $p_d(0)=p_d(-1)=0$.  Since  
$f_{{\rm sd}^n(\Delta)}(-1)= f_\Delta(-1) = - \tilde{\chi}(\Delta)$, we see by \eqref{pblimsdn} that  for all $n$ sufficiently large all zeros of $f_{{\rm sd}^n(\Delta)}(x)$ are simple and lie in $[-1,0)$ (since $f_{{\rm sd}^n(\Delta)}(x)$ has the correct sign to the left of $-1$). By \eqref{pbhpol} this is equivalent to (2). Statement (3)  follows similary. 
 \end{proof}

\begin{corollary}
Let $\Delta$ be a $(d-1)$-dimensional Boolean cell complex such that $h_\Delta(x)$  is symmetric and $(-1)^{d-1}\tilde{\chi}(\Delta) \geq 0$. Then there is a number $N(\Delta)$ such that $h_{{\rm sd}^n(\Delta)}(x)$ is $\gamma$-nonnegative whenever $n \geq N(\Delta)$. 
\end{corollary}

\begin{proof}
Combine Remark \ref{pbgammaroots}, Lemma \ref{pbsdsym} and Corollary \ref{pblimcor}. 
 \end{proof}

\subsection{Unimodality of $h^*$-polynomials}
Let $P \subset \mathbb{R}^n$ be an $m$-dimensional \emph{integral polytope}, i.e., all vertices have integer coordinates. Ehrhart \cite{pbEhr1,pbEhr2} proved that the function 
$$
i(P,r) = |rP \cap \mathbb{Z}^n|,
$$ 
which counts the number of integer points in the $r$-fold dilate of $P$, is a polynomial in $r$ of degree $m$. It follows that we may write 
\begin{equation}\label{pberh}
\sum_{r=0}^\infty i(P,r) x^r = \frac { h_0^*(P) + h_1^*(P)x+\cdots+ h_m^*(P)x^m}{(1-x)^{m+1}}. 
\end{equation}
Stanley \cite{pbStandec} proved that the coefficients of the  polynomial, $h^*_P(x)$, in the numerator of \eqref{pberh} are nonnegative, and Hibi \cite{pbHibibook} conjectured that $h^*_P(x)$ is unimodal whenever it is symmetric. Hibi \cite{pbHibibook} proved the conjecture for $n \leq 5$. However Payne and Musta{\c{t}}{\v{a}} \cite{pbMustPay,pbPayne} found counterexamples to Hibi's conjecture for each $n \geq 6$. Let us mention a weaker conjecture that is still open. An integral polytope $P$ is \emph{Gorenstein} if $h^*_P(x)$ is symmetric, and  $P$ is \emph{integrally closed} if each integer point in $rP$ may be written as a sum of $r$ integer points in $P$, for all $r \geq 1$. 
\begin{conjecture}[Ohsugi--Hibi, \cite{pbOhHi}]
If $P$ is a Gorenstein and integrally closed integral polytope, then $h^*_P(x)$ is unimodal. 
\end{conjecture}
Inspired by work of Reiner and Welker \cite{pbReWe}, Athanasiadis \cite{pbAth05} provided  conditions on an integral polytope $P$ which imply that $h^*_P(x)$ is the $h$-polynomial of the boundary complex of a simplicial polytope. Hence, by the $g$-theorem (see  \cite{pbStanComCom}), $h^*_P(x)$ is unimodal. Athanasiadis used this result to prove the following conjecture of 
Stanley. An \emph{integer stochastic matrix} is a square matrix with nonnegative integer entries having all row- and column sums equal to each other. Let $H_n(r)$ be the number of $n\times n$ integer stochastic matrices with row- and column sums equal to $r$. The function $r \mapsto H_n(r)$  is the Ehrhart polynomial of the integral polytope $P_n$ of real doubly stochastic matrices. Stanley \cite{pbStanComCom} conjectured that $h^*_{P_n}(x)$ is unimodal for all positive integers $n$, and Athanasiadis' proof of Stanley's conjecture was the main application of the techniques developed in \cite{pbAth05}.  Subsequently Bruns and R\"omer \cite{pbBrRo} generalized Athanasiadis results to the following general theorem. 
\begin{theorem}
Let $P$ be a Gorenstein integral polytope such that $P$ has a regular unimodular triangulation. Then $h^*_P(x)$ is the $h$-polynomial of the boundary complex of a simplicial polytope. In particular, $h^*_P(x)$ is unimodal.
\end{theorem}

\section{Log--concavity and matroids}
Several important sequences associated to matroids have been conjectured to be log--concave. Progress on these conjectures have been very limited until the recent breakthrough of Huh and Huh--Katz \cite{pbHuh,pbHuhKatz}. Recall that the \emph{characteristic polynomial} of a matroid $M$ is defined as 
$$
\chi_M(x)= \sum_{F \in L_M} \mu(\hat{0}, F) x^{r(M)-r(F)}= \sum_{k=0}^r (-1)^kw_k(M)x^{r(M)-k},
$$
where $L_M$ is the lattice of flats, $\mu$ its M\"obius function, $r$ is the rank function of $M$ and $\{(-1)^kw_k(M)\}_{k=0}^{r(M)}$ are the \emph{Whitney numbers of the first kind}. The sequence $\{w_k(M)\}_{k=0}^r$ is nonnegative, and it was conjectured by Rota and Heron to be unimodal. Welsh later conjectured that $\{w_k(M)\}_{k=0}^{r(M)}$ is log--concave. It is known that $\chi_M(1)=0$. Define the \emph{reduced characteristic polynomial} by 
$$
\bar{\chi}_M(x)= \chi_M(x)/(x-1)=: \sum_{k=0}^{r-1}(-1)^kv_k(M) x^{r(M)-1-k}.
$$
Note that if $\{v_k(M)\}_{k=0}^{r(M)-1}$ is log--concave, then so is $\{w_k(M)\}_{k=0}^{r(M)}$, see \cite{pbStanS}. 
\begin{theorem}[Huh--Katz, \cite{pbHuhKatz}]\label{pbHHK}
If $M$ is representable over some field, then the sequence $\{v_k(M)\}_{k=0}^{r(M)-1}$ is log--concave. 
\end{theorem}
Since the \emph{chromatic polynomial} of a graph is the characteristic polynomial of a representable matroid we have the following corollary:
\begin{corollary}[Huh, \cite{pbHuh}]
Chromatic polynomials of graphs are log--concave.
\end{corollary}
Let 
$$
f_M(x)= \sum_{k=0}^{r(M)} (-1)^k f_k(M) x^{r(M)-k},
$$
where $f_k(M)$ is the number of independent sets of $M$ of cardinality $k$. Hence $f_M(x)$ is the (signed)  $f$-polynomial of the independence complex of $M$. Now, $f_M(x)= \bar{\chi}_{M\times e}(x)$, where $M\times e$ is the \emph{free coextension} of $M$, see \cite{pbBry, pbLenz}. Also if $M$ is representable over some field, then so is $M\times e$. Hence the following corollary is a consequence of Theorem \ref{pbHHK}. 
\begin{corollary}
If $M$ is representable over some field, then $\{f_k(M)\}_{k=0}^{r(M)}$ is log--concave.
\end{corollary}
This corollary, first noted by Lenz \cite{pbLenz}, verifies  the weakest version of Mason's conjecture below for the class of representable matroids.
\begin{conjecture}[Mason]
Let $M$ be a matroid and $n=f_1(M)$. The following sequences are log--concave:
$$
\{f_k(M)\}_{k=0}^{r(M)}, \ \ \ \{k!f_k(M)\}_{k=0}^{r(M)}, \ \ \ and \ \ \ \left\{f_k(M)/\binom n k\right\}_{k=0}^{r(M)}.
$$
\end{conjecture}

The proofs in \cite{pbHuh,pbHuhKatz} use involved algebraic machinery which falls beyond the scope of this survey. It is unclear if the method can be extended to the case of non--representable matroids.

\section{Infinite log-concavity}
Consider the operator $\mathcal{L}$ on sequences $\mathcal{A}=\{a_k\}_{k=0}^\infty \subset \mathbb{R}$ defined by $\mathcal{L}(\mathcal{A})=\{b_k\}_{k=0}^\infty$, where 
$$
b_0= a_0^2 \ \ \mbox{ and } \ \ b_k = a_k^2-a_{k-1}a_{k+1}, \ \  \mbox{ for  } k \geq 1.
$$
This definition makes sense for finite sequences by regarding these as infinite sequences with finitely many nonzero entries.  Hence a sequence $\mathcal{A}$ is log--concave if and only if $\mathcal{L}(\mathcal{A})$ is a nonnegative sequence. A sequence is $k$-\emph{fold log-concave} if $\mathcal{L}^j(\mathcal{A})$ is a nonnegative sequence for all $0\leq j \leq k$. A sequence is \emph{infinitely log-concave} if it is $k$-fold log-concave for all $k \geq 1$. Although similar notions were studied by Craven and Csordas \cite{pbCrCsJP,pbCrCsIt}, the following questions asked by Boros and Moll \cite{pbBoMo} spurred the interest in infinite log-concavity in the combinatorics community:
\begin{itemize} 
\item[(A)] For $m \in \mathbb{N}$, let $\{d_\ell(m)\}_{\ell=0}^m$ be defined by 
$$
d_\ell(m)= 4^{-m}\sum_{k=\ell}^m 2^k \binom {2m-2k}{m-k} \binom {m+k}{m} \binom k \ell .
$$
Is the sequence $\{d_\ell(m)\}_{\ell=0}^m$ infinitely log-concave? 
\item[(B)] For $n \in \mathbb{N}$, is the sequence $\{\binom n k\}_{k=0}^n$ infinitely log-concave? 
\end{itemize}
Question (A) is still open. However Chen \emph{et.~al.} \cite{pbChen1} proved $3$-fold log-concavity of $\{d_\ell(m)\}_{\ell=0}^m$ by proving a related conjecture of the author which implies $3$-fold log-concavity of $\{d_\ell(m)\}_{\ell=0}^m$, for each $m \in \mathbb{N}$, by the work of Craven and Csordas \cite{pbCrCsIt}. 

In connection to (B), Fisk \cite{pbFisk1}, McNamara--Sagan \cite{pbMcSa}, and Stanley \cite{pbStanPC} independently conjectured the next theorem from which (B) easily follows. We may consider $\mathcal{L}$ to be an operator on the generating function of the sequence, i.e., 
$$
\mathcal{L}\left(\sum_{k=0}^\infty a_kx^k\right)= \sum_{k=0}^\infty (a_k^2-a_{k-1}a_{k+1})x^k.
$$

\begin{theorem}[\cite{pbBrIt}]\label{pbite}
If $f(x)=\sum_{k=0}^n a_k x^k$ is a polynomial with real- and nonpositive zeros only, then so is $\mathcal{L}(f)$. In particular, the sequence $\{a_k\}_{k=0}^n$ is infinitely log-concave. 
\end{theorem}

The proof of Theorem \ref{pbite} uses multivariate techniques, and will be given in Section \ref{pbSecGWS}. 

There is a simple criterion on a nonnegative sequence  $\mathcal{A}=\{a_k\}_{k=0}^\infty$ that guarantees infinite log--concavity \cite{pbCrCsIt, pbMcSa}. Namely 

$$
a_k^2 \geq r a_{k-1}a_{k+1},  \ \ \ \mbox{ for all } k \geq 1,
$$
where $r \geq (3+\sqrt{5})/2$.

McNamara and Sagan \cite{pbMcSa} conjectured that the operator $\mathcal{L}$ preserves the class of PF sequences. In particular they conjectured that the columns of  Pascal's triangle $\{ \binom {n+k} k \}_{n=0}^\infty$, where $k \in \mathbb{N}$, are infinitely log--concave. In \cite{pbBrCh}, Chasse and the author found counterexamples to the first mentioned conjecture and proved the second. They considered   PF sequences that are interpolated by polynomials, i.e., PF sequences $\{p(k)\}_{k=0}^\infty$ where $p$ is a polynomial, and asked when classes of such sequences are preserved by $\mathcal{L}$. 

Let $\mathcal{P}$ be the following class of PF sequences which are interpolated by polynomials 
$$
\big\{ \{p(k)\}_{k=0}^\infty \in {\rm PF}: p(x) \in \mathbb{R}[x] \mbox{ and } p(-j) = p(-j+1)=0 \mbox{ for some } j \in \{0,1,2\} \big\}. 
$$

\begin{theorem}[ \cite{pbBrCh}]
The operator $\mathcal{L}$ preserves the class $\mathcal{P}$. In particular each sequence in $\mathcal{P}$ is infinitely log--concave. 
\end{theorem}
Note that for each $k \in \mathbb{N}$, 
$\{ \binom {n+k} k \}_{n=0}^\infty \in \mathcal{P}$. The following corollary solves the above mentioned conjecture of McNamara and Sagan.
\begin{corollary}
The columns of Pascal's triangle are infinitely log--concave, i.e., for each $k \in \mathbb{N}$, the sequence 
$\{ \binom {n+k} k \}_{n=0}^\infty$ is infinitely log--concave. 
\end{corollary}

Let us end this section with an interesting open problem posed by Fisk \cite{pbFisk1}. 
\begin{problem}\label{pbFiskprob}
Suppose all zeros of  $\sum_{k=0}^n a_k x^k$ are nonpositive. If $d \in \mathbb{N}$, are all zeros of 
$$
\sum_{k=0}^n \det (a_{k+i-j})_{i,j=0}^d \cdot x^k, 
$$
where $a_i =0$ if $i \not \in \{0,\ldots, n\}$, nonpositive?
\end{problem}
Hence the case $d=1$ of Problem \ref{pbFiskprob} is Theorem \ref{pbite}. 

\section{The Neggers--Stanley conjecture}
It is natural to ask if the real--rootedness of the Eulerian polynomials may be extended to generating polynomials of linear extensions of any poset. Define a \emph{labeled poset} to be a poset of the form $P=([n], \leq_P)$, where $n$ is a  positive integer. The \emph{Jordan--H\"older} set of $P$,
$$
\mathfrak{L}(P)= \{ \sigma \in \mathfrak{S}_n : i<j \mbox{ whenever } \sigma(i) <_P \sigma(j)\}, 
$$
is the set of all linear extensions of $P$. Here $<$ denotes the usual order on the integers. 
The $P$--\emph{Eulerian polynomial} is defined by 
$$
W_P(x)= \sum_{\sigma \in \mathfrak{L}(P)} x^{{\rm des}(\sigma)+1}.
$$
Recall that $P$ is \emph{naturally labeled} if $i < j$ whenever $i<_P j$. Neggers \cite{pbNeg} conjectured in $1978$ that $W_P(x)$ is real--rooted for any naturally labeled poset $P$, and Stanley extended the conjecture to all labeled posets in $1986$, see \cite{pbBreTh, pbBreS, pbWa2}.  Counterexamples to Stanley's conjecture were first found by the author in \cite{pbCounter}, and shortly thereafter naturally labeled counterexamples were found by Stembridge in \cite{pbStem1}, see Fig.~\ref{pbcountpic}. 
\begin{figure}
\begin{center}
 \includegraphics[height=5cm, width=6cm]{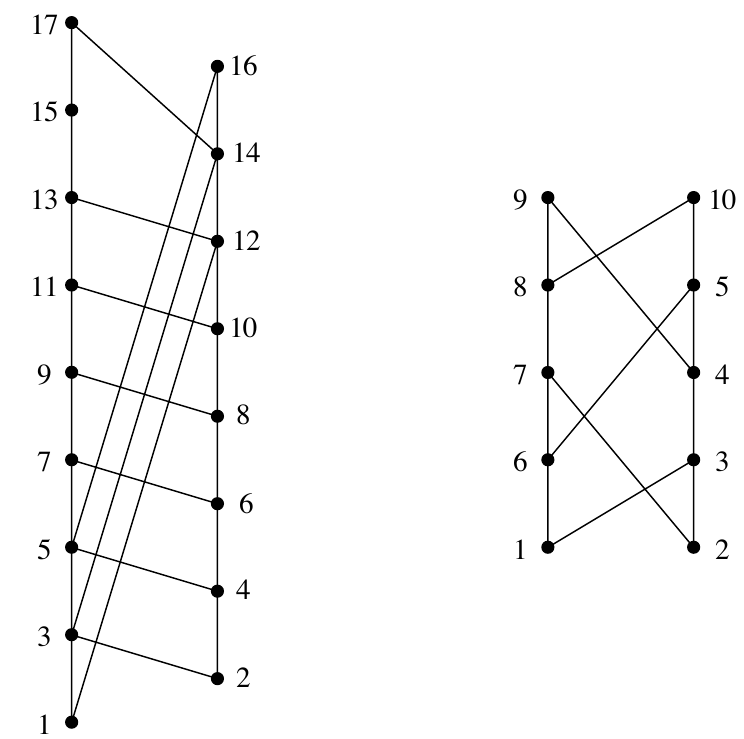}
\end{center}
\caption{Counterexamples to  Neggers conjecture (left) and the Neggers--Stanley conjecture (right), taken from \cite{pbStem1}.}\label{pbcountpic}
\end{figure} 

However, this does not seem to be the end of the story. Recall that a poset $P$ is \emph{graded} if all maximal chains in $P$ have the same size. 

\begin{theorem}[Reiner and Welker, \cite{pbReWe}]\label{pbrw}
If $P$ is a graded and naturally labeled poset, then  $W_P(x)$ is unimodal. 
\end{theorem}
Reiner and Welker proved Theorem \ref{pbrw} by associating to $P$ a simplicial polytope whose $h$--polynomial is equal to $W_P(x)$, and then invoking the $g$--theorem for simplicial polytopes. 

Theorem \ref{pbrw} was refined in \cite{pbSign,pbAct} to establish $\gamma$--nonnegativity for the $P$--Eulerian polynomials of a class of  labeled posets which contain the graded and naturally labeled posets. Let $E(P)=\{(i,j) : j \mbox{ covers } i\}$ be the \emph{Hasse diagram} of a labeled poset $P$. Define a function $\epsilon : E(P) \rightarrow \{-1,1\}$, by 
$$
\epsilon(i,j)= 
\begin{cases}
1 &\mbox{ if } i<j, \mbox{ and }\\
-1 &\mbox{ if } j<i.
\end{cases}
$$
A labeled poset $P$ is \emph{sign--graded} if for all maximal chains $x_0<_P x_1 <_P \cdots <_P x_k$ in $P$, the quantity 
$$
r=\sum_{i=1}^k\epsilon(x_{i-1},x_i)
$$
is the same, see Fig~\ref{pbpicsig}. Note that a naturally labeled poset is sign--graded if and only if it is graded. 
\begin{figure}[htp]
$$
\begin{tikzpicture}[scale=0.75, inner sep=2pt]
  \node (2) at (0,0) {$2$};
  \node (6) at (0,1) {$6$};
    \node (3) at (0,2) {$3$};
    \node (7) at (0,3) {$7$};
    \node (1) at (2,0) {$1$};
    \node (5) at (2,1) {$5$};
  \node (4) at (2,2) {$4$};
  \node (8) at (2,3) {$8$};
  \node (9) at (3,1) {$9$};
  \draw (2) -- (6) -- (3) -- (7) -- (4) -- (5) -- (1) -- (9);
  \draw (2) -- (8) -- (4);
\end{tikzpicture}
$$ 
\caption{ A sign--graded poset of rank $1$. } 
\label{pbpicsig}
\end{figure}
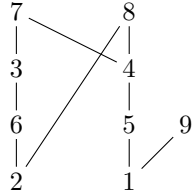

\begin{theorem}\label{pbsgrd}
If $P$ is sign--graded, then $W_P(x)$ is $\gamma$--nonnegative. 
\end{theorem}

Two proofs are known for Theorem \ref{pbsgrd}. The first proof \cite{pbSign} uses a partitioning of $\mathfrak{L}(P)$ into Jordan--H\"older sets of refinements of $P$ for which $\gamma$--positivity is easy to prove. The second proof \cite{pbAct} uses an extension to $\mathfrak{L}(P)$ of the $\mathbb{Z}_2^n$-action described in Section \ref{pbSecAct}. 

Here are two questions left open regarding the Neggers--Stanley conjecture.

\begin{question}
Are the coefficients of $P$--Eulerian polynomials log--concave or unimodal?
\end{question} 

\begin{question}
Are $P$--Eulerian polynomials of graded (or sign--graded) posets real--rooted?
\end{question} 

The work in \cite{pbSign} was generalized by Stembridge \cite{pbStem2} to certain Coxeter cones. Let $\Phi$ be a  finite root system in a real Euclidian space $V$ with inner product $\langle \cdot, \cdot \rangle$. A \emph{Coxeter cone} is a closed convex cone of the form 
$$
\Delta(\Psi) = \{ \mu \in V: \langle \mu, \beta \rangle \geq 0 \mbox{ for all } \beta \in \Psi\},
$$
where $\Psi \subseteq \Phi$. This cone is a closed union of cells of the Coxeter complex defined by $\Phi$, so it forms a simplicial complex which we identify with $\Delta(\Psi)$. A \emph{labeled Coxeter cone} is a cone of the form 
$$
\Delta(\Psi,\lambda)= \{\mu \in \Delta(\Psi) : \langle \mu, \beta \rangle > 0 \mbox{ for all } \beta \in \Psi \mbox{ with } \langle \lambda, \beta \rangle < 0\}, 
$$
where $\Delta(\Psi)$ is a Coxeter cone and $\lambda \in V$.  Hence $\Delta(\Psi,\lambda)$ may be identified with a \emph{relative} complex inside $\Delta(\Psi)$. When $\Phi$ is crystallographic,  Stembridge defines what it means for a (labeled) Coxeter cone to be graded. In type $A$, graded labeled Coxeter cones correspond to sign--graded posets. 

\begin{theorem}[Stembridge, \cite{pbStem2}]
The $h$-vectors of graded labeled Coxeter cones are $\gamma$-nonnegative. 
\end{theorem}

\section{Preserving real--rootedness}\label{pbSecPres}
If a sequence of polynomials satisfies a linear recursion, then to prove that the polynomials are real--rooted it is sufficient to prove that the defining recursion ``preserves" real--rootedness. Hence it is natural, from a combinatorial point of view,  to ask which linear operators on polynomials preserve real--rootedness. This question has a rich history that goes back to the work of Jensen, Laguerre and P\'olya, see the survey \cite{pbCrCsS}. In his thesis, Brenti \cite{pbBreTh} studied this question focusing on operators occurring naturally in combinatorics. 

Let us recall P\'olya and Schur's \cite{pbPolyaSchur} celebrated characterization of \emph{diagonal} operators preserving real--tootedness. A sequence $\Lambda= \{\lambda_k \}_{k=0}^\infty$ of real numbers is called a \emph{multiplier sequence (of the first kind)}, if the linear operator $T_\Lambda : \mathbb{R}[x] \rightarrow \mathbb{R}[x]$ defined by 
$$
T_\Lambda(x^k)= \lambda_k x^k, \ \ \ \ k \in \mathbb{N},
$$
preserves real--rootedness. 

\begin{theorem}[P\'olya and Schur, \cite{pbPolyaSchur}]\label{pbPSThm}
Let  $\Lambda= \{\lambda_k \}_{k=0}^\infty$ be a sequence of real numbers, and let 
$$
G_\Lambda(x)= \sum_{k=0}^\infty \frac {\lambda_k} {k!} x^k,
$$
be its exponential generating function. The following assertions are equivalent:
\begin{itemize}
\item[(1)] $\Lambda$ is a multiplier sequence.
\item[(2)] For all nonnegative integers $n$, the polynomial 
$$
T\big((x+1)^n\big) = \sum_{k=0}^n \binom n k \lambda_k x^k,
$$ 
is real--rooted, and all its zeros have the same sign. 
\item[(3)] Either $G_\Lambda(x)$ or $G_\Lambda(-x)$ defines an entire function that can be written as 
$$
G_\Lambda(\pm x) = Cx^m e^{ax} \prod_{k=1}^\infty (1+\alpha_k x),
$$
where $m \in \mathbb{N}$, $C  \in \mathbb{R}$, $a \geq 0$, $\alpha_k \geq 0$ for all $k \in \mathbb{N}$, and $\sum_{k=1}^{\infty} \alpha_k <\infty$. 
\item[(4)] $G_\Lambda(x)$ defines an entire function which is the limit, uniform on compact subsets of $\mathbb{C}$, of real--rooted polynomials whose zeros all have the same sign.  
\end{itemize}
\end{theorem}

\begin{example}\label{pbBesselex}
Let $\Lambda = \{1/k!\}_{k=0}^\infty$. Then $T_\Lambda\big((x+1)^n)\big) =L_n(-x)$, where 
$L_n(x)$ is the $n$th \emph{Laguerre polynomial}. Since orthogonal polynomials are real--rooted (REF) we see that (2) of Theorem \ref{pbPSThm} is satisfied, and thus $\Gamma$ is a multiplier sequence. 
\end{example}

Only recently a complete characterization of linear operators preserving real--rootedness was obtained by Borcea and the author in \cite{pbAnn}. This characterization is in terms of a natural extension of real--rootedness to several variables. A polynomial $P(x_1,\ldots, x_m) \in \mathbb{C}[x_1,\ldots, x_m]$  is called \emph{stable} if 
$$
{\rm Im}(x_1)>0, \ldots, {\rm Im}(x_m)>0 \ \ \mbox{ implies } \ \ P(x_1,\ldots, x_n) \neq 0.
$$
By convention we also consider the identically zero polynomial to be stable. Hence a univariate real polynomial is stable if and only if it is real--rooted. Let 
$\alpha_1 \leq \cdots \leq \alpha_n$ and $\beta_1 \leq \cdots \leq \beta_m$ be the zeros of two real--rooted polynomials. We say that these zeros \emph{interlace} if 
$$
\alpha_1 \leq \beta_1 \leq \alpha_2 \leq \beta_2 \leq \cdots \ \ \mbox{ or } \ \ \beta_1 \leq \alpha_1 \leq \beta_2 \leq \alpha_2 \leq \cdots  .$$
By convention, the ``zeros'' of any two polynomials of degree $0$ or $1$ interlace. Interlacing zeros is characterized by a linear condition as the following theorem which is often attributed to Obreschkoff describes:
\begin{theorem}[Satz  5.2 in \cite{pbObr}]\label{pbobr}
Let $f, g \in \mathbb{R}[x]\setminus\{0\}$. Then the zeros of $f$ and $g$ interlace if and only if all polynomials in the linear space 
$$
\{\alpha f + \beta g : \alpha, \beta \in \mathbb{R}\}
$$
are real--rooted. 

\end{theorem} 

Let $\mathbb{R}_n[x] = \{ p \in \mathbb{R}[x] : \deg p \leq n\}$. The \emph{symbol} of a linear operator 
$T: \mathbb{R}_n[x] \rightarrow \mathbb{R}[x]$ is the bivariate polynomial 
$$
G_T(x,y)= T\big((x+y)^n\big) := \sum_{k=0}^n \binom n k T(x^k) y^{n-k} \in \mathbb{R}[x,y].
$$
\begin{theorem}[\cite{pbAnn}]\label{pbalgchar}
Let $T : \mathbb{R}_n[x] \rightarrow \mathbb{R}[x]$ be a linear operator. Then $T$ preserves real--rootedness if and only if (1), (2) or (3) below is satisfied. 
\begin{itemize}
\item[(1)] $T$ has rank at most two and is of the form 
$$
T(p)=\alpha(p)f + \beta(p)g,
$$
where $\alpha, \beta : \mathbb{R}[x] \rightarrow \mathbb{R}$ are linear functionals and $f,g$ are real--rooted polynomials whose zeros interlace. 
\item[(2)] $G_T(x,y)$ is stable. 
\item[(3)] $G_T(x,-y)$ is stable. 
\end{itemize}
\end{theorem}
\begin{example}
The operators of type (1) are the ones achieved by Theorem \ref{pbobr}. An example of an operator of type (2) is $T= d/dx$, because then 
$G_T(x,y)= n(x+y)^{n-1}$. An example of an operator of type (3) is the algebra automorphism, $S : \mathbb{R}[x] \rightarrow \mathbb{R}[x]$,  defined by $S(x)=-x$. Indeed $T$ is of type (2) if and only if $T \circ S$ is of type (3). 
\end{example} 

To illustrate how Theorem \ref{pbalgchar} may be used let us give a simple example from combinatorics.
\begin{example}
The Eulerian polynomials satisfy the recursion 
$A_{n+1}(x) = T_n(A_n(x))$, where 
$$
T_n= x(1-x)\frac d {dx} +(n+1)x,
$$
see \cite{pbStanEn}. The symbol of $T_n : \mathbb{R}_n[x] \rightarrow \mathbb{R}[x]$ is 
$$
T_n\big((x+y)^n\big) = x(x+y)^{n-1}(x+(n+1)y+n),
$$
which is trivially  stable. Hence $A_n(x)$ is real--rooted for all $n \in \mathbb{N}$ by Theorem \ref{pbalgchar}. This was first proved by Frobenius \cite{pbFrob}. 
\end{example}
A characterization of stable polynomials in two variables --- and hence of the symbols of preservers of real--rootedness --- follows from  Helton and  Vinnikov's characterization of real--zero polynomials in \cite{pbHV}, see \cite{pbLondon}. 
\begin{theorem}\label{pbHV-thm}
Let $P(x,y) \in \mathbb{R}[x,y]$ be a polynomial of degree $d$. Then $P$ is stable if and only if there exists three real symmetric $d \times d$ matrices $A,B$ and $C$ and a  real number $r$ such that 
$$
P(x,y)= r\cdot \det(xA+yB+C),
$$ 
where $A$ and $B$ are positive semidefinite and $A+B$ is the identity matrix. 
\end{theorem} 
For the unbounded degree analog of Theorem \ref{pbalgchar} we define the \emph{symbol} of a linear operator $T : \mathbb{R}[x] \rightarrow \mathbb{R}[x]$ to be the 
formal powers series 
$$
\bar{G}_T(x,y)= T(e^{-xy}) := \sum_{n=0}^\infty (-1)^n \frac{T(x^n)}{n!} y^n \in \mathbb{R}[x][[y]].
$$
The \emph{Laguerre--P\'olya class}, $\mathcal{L\textendash P}_n$, is defined to be the class of real entire functions in $n$ variables which are the uniform limits on compact subsets of $\mathbb{C}$ of real stable polynomials. For example $\exp(-x_1x_2-x_3x_4+2x_5) \in \mathcal{L\textendash P}_5$ since it is the limit of the stable polynomials
$$
\left(1-\frac {x_1x_2} n \right)^n \left(1-\frac {x_3x_4} n \right)^n\left(1+2\frac {x_5} n \right)^n.
$$ 
\begin{theorem}[\cite{pbAnn}]\label{pbtrans}
Let $T : \mathbb{R}[x] \rightarrow \mathbb{R}[x]$ be a linear operator. Then $T$ preserves real--rootedness if and only if (1), (2) or (3) below is satisfied. 
\begin{itemize}
\item[(1)] $T$ has rank at most two and is of the form 
$$
T(p)=\alpha(p)f + \beta(p)g,
$$
where $\alpha, \beta : \mathbb{R}[x] \rightarrow \mathbb{R}$ are linear functionals and $f,g$ are real--rooted polynomials whose zeros interlace. 
\item[(2)] $\bar{G}_T(x,y) \in \mathcal{L\textendash P}_2$. 
\item[(3)] $\bar{G}_T(x,-y) \in \mathcal{L\textendash P}_2$. 
\end{itemize}
\end{theorem}

There are, as of yet, no analogs of Theorems \ref{pbalgchar} and \ref{pbtrans} for linear operators that preserve the property of having all zeros in a prescribed interval (other than $\mathbb{R}$ itself). 

\begin{problem}\label{pbintprob}
Let $I \subset \mathbb{R}$ be an interval. Characterize all linear operators on polynomials that preserve the property of having all zeros in $I$. 
\end{problem}

For polynomials appearing in combinatorics the case when $I=(-\infty, 0]$ is the most  important.

\subsection{The subdivision operator}\label{pbSecSO}

An example of an operator of the kind appearing in Problem \ref{pbintprob} is the ``subdivision'' operator  $\mathcal{E} : \mathbb{R}[x] \rightarrow \mathbb{R}[x]$  in Section \ref{pbSecGh}. 
 The following theorem by Wagner proved the Neggers--Stanley conjecture for series--parallel posets, see \cite{pbWa2, pbWa1}. 
\begin{theorem}[\cite{pbWa1}]\label{pbwag}
If all zeros of $\mathcal{E}(f)$ and $\mathcal{E}(g)$ lie in the interval $[-1,0]$, then so does the zeros of 
$\mathcal{E}(fg)$. 
\end{theorem}
As we have seen in Section \ref{pbSecGh}, the next theorem has consequences in topological combinatorics. 
\begin{theorem}[\cite{pbTrans}]\label{pbcons}
If
$$
f(x) = \sum_{k=0}^d h_k x^k(1+x)^{d-k},
$$
where $h_k \geq 0$ for all $0\leq k \leq d$,  then all zeros of $\mathcal{E}(f)$ are real, simple and located in $[-1,0]$. 
\end{theorem} 
The main part of the next theorem was proved by Brenti and Welker in \cite{pbBrWe}, while (2) was proved in \cite{pbDeps} . We take the opportunity to give alternative simple proofs below.
\begin{theorem}\label{pbeigen}
For each integer $n \geq 2$, $\mathcal{E}$ has a unique monic eigen-polynomial, $p_n(x)$, of degree $n$. 

Moreover, 
\begin{itemize}
\item[(1)] all zeros of $p_n(x)$ are real, simple and lie in the interval  $[-1,0]$; 
\item[(2)] $p_n(x)$ is symmetric around $-1/2$, i.e., 
$$
(-1)^np_n(-1-x)= p_n(x).
$$
\end{itemize}
\end{theorem}
\begin{proof}
Let $n \geq 2$. Consider the map $\phi : [-1,0]^n \rightarrow [-1,0]^n$ defined as follows. Let $\theta = (\theta_1, \ldots, \theta_n) \in [-1,0]^n$. Since $\mathcal{E}$ preserves the property of having all zeros in $[-1,0]$ (Theorems \ref{pbwag} and \ref{pbcons}), we may order the zeros of  $\mathcal{E}((x-\theta_1) \cdots (x-\theta_n))$ as $-1 \leq \alpha_1 \leq \cdots \leq \alpha_n \leq 0$. Let $\phi(\theta):= (\alpha_1, \ldots, \alpha_n)$.
 By Hurwitz' theorem on the continuity of zeros \cite{pbRS}, $\phi$ is continuous. Hence by Brouwer's fixed point theorem $\phi$ has a fixed point, which then corresponds to a degree $n$ eigen-polynomial, $p_n$, of $\mathcal{E}$. It follows by examining the leading coefficients that the corresponding eigenvalue is $n!$. 

Set $p_0 :=1$ and $p_1:= x+1/2$. Let $f$ be an arbitrary monic polynomial of degree $n \geq 2$, and let $T= n!^{-1}\mathcal{E}$. Now by expanding $f$ as a linear combination of $\{p_k\}_{k=0}^n$,
$$
f= \sum_{i=0}^n a_i p_i,
$$ 
 we see that 
$$
\lim_{k \to \infty} T^k(f) = \lim_{k \to \infty} \sum_{i=0}^n \left( \frac {i!} {n!} \right)^k \!\! a_i p_i = p_n, 
$$
since  $a_n=1$. Hence $p_n$ is unique. By choosing $f$ to be $[-1,0]$--rooted, we see that $T^k(f)$ is also $[-1,0]$-rooted for all $k$. By Hurwitz' theorem, so is $p_n$. Since $p_n$ is $[-1,0]$--rooted, it is certainly of the form displayed in Theorem \ref{pbcons}. By Theorem \ref{pbcons} again, the zeros of $p_n= n!^{-1}\mathcal{E}(p_n)$ are distinct. 

Property (2) follows immediately from \eqref{pbSE}. 
 \end{proof}

It is easy to see that the coefficients of $p_n(x)$ are rational numbers for each $n \geq 2$. 
\begin{question}
Is there a closed formula for $p_n(x)$? What are the generating functions
$$
A(x,y) = \sum_{n = 0}^\infty p_n(x) y^n \ \ \mbox{ and } \ \ B(x,y) = \sum_{n = 0}^\infty \frac {p_n(x)} {n!} y^n ? 
$$
Note that $\mathcal{E}(B) = A$. 
\end{question}

\section{Common interleavers}\label{pbSecCI}
A powerful technique for proving that families of polynomials are real--rooted  is that of \emph{compatible polynomials}. This was employed by Chudnovsky and Seymour \cite{pbChSe} to prove a conjecture of Hamidoune and Stanley on the zeros of independence polynomials of clawfree graphs. Subsequently an elegant alternative proof was given by Lass \cite{pbLass}, by proving a Mehler formula for independence polynomials of clawfree graphs.  An \emph{independent set}  in a finite and simple graph $G=(V,E)$ is a set of pairwise non-adjacent vertices. The \emph{independence polynomial} of $G$ is the polynomial 
$$
I(G,x)= \sum_S x^{|S|},
$$
where the sum is over all independent sets $S \subseteq V$. Recall that a \emph{claw} is a graph isomorphic to the graph on $V=\{1,2,3,4\}$ with edges $E=\{12,13,14\}$.  Note that the independence polynomial of a claw is $1+4x+3x^2+x^3$, which has two non--real zeros. A graph is clawfree if no induced subgraph is a claw. The next theorem was posed as a question by Hamidoune  \cite{pbHam} and later as a conjecture by Stanley \cite{pbStanChr}. 
\begin{theorem}[\cite{pbChSe,pbLass}]\label{pbClaw}
If $G$ is a clawfree graph, then all zeros of $I(G,x)$ are real.  
\end{theorem}

Let $f,g  \in \mathbb{R}[x]$ be two real--rooted polynomials with positive leading coefficients. We say that $f$ is an \emph{interleaver} of $g$ (written $f \ll g$) if 
$$
        \cdots \leq \alpha_2 \leq \beta_2 \leq \alpha_1 \leq  \beta_1,
$$
where $\{\alpha_i\}_{i=1}^n$ and $\{\beta_i\}_{i=1}^m$ are the zeros of $f$ and $g$, respectively.  By convention we also write $0 \ll 0$, $0 \ll h$ and $h \ll 0$, where $h$ is any real--rooted polynomial with positive leading coefficient. If $f \ll g$ and $f \not \equiv 0$, we say that $f$ is a \emph{proper interleaver} of $g$.  The polynomials $f_1(x), \ldots, f_m(x)$ are 
$k$--\emph{compatible}, where $1\leq k \leq m$, if 
$$
\sum_{j \in S} \lambda_j f_j(x) 
$$
is real--rooted whenever $S \subseteq  [m]$, $|S|=k$ and $\lambda_j \geq 0$ for all $j \in S$. The following theorem was used in Chudnovsky and Seymour's proof of Theorem \ref{pbClaw}. 

\begin{theorem}[Chudnovsky--Seymour, \cite{pbChSe}]\label{pbCS}
Suppose that the leading coefficients of  $f_1(x), \ldots, f_m(x) \in \mathbb{R}[x]$ are positive. The following are equivalent. 
\begin{enumerate}
\item $f_1(x), \ldots, f_m(x)$ are $2$-compatible;  
\item For all $1\leq i < j \leq m$, $f_i(x)$ and $f_j(x)$ have a proper common interleaver;  
\item $f_1(x), \ldots, f_m(x)$ have a proper common interleaver;
\item $f_1(x), \ldots, f_m(x)$ are $m$-compatible.  
\end{enumerate}
\end{theorem}

Theorem \ref{pbCS} is useful in situations when the polynomials of interest may be expressed as a nonnegative sums of similar polynomials. In order to prove that the polynomials of interest are real--rooted it then suffices to prove that the similar polynomials are $2$--compatible. 

A sequence $F_n= (f_i)_{i=1}^n$ of real--rooted polynomials is called \emph{interlacing} if $f_i \ll f_j$ for all  $1\leq i<j\leq n$. Let $\mathcal{F}_n$ be the family of all interlacing 
sequences $(f_i)_{i=1}^n$ of polynomials, and let $\mathcal{F}_n^+$ be the family of  $(f_i)_{i=1}^n \in \mathcal{F}_n$ such that $f_i$ has nonnegative coefficients for all $1 \leq i \leq n$. We are are interested in when an $m \times n$ matrix $G= (G_{ij}(x))$ of polynomials maps $\mathcal{F}_n$ to $\mathcal{F}_m$ (or $\mathcal{F}_n^+$ to $\mathcal{F}_m^+$) by the action 
$$
G\cdot F_n = (g_1, \ldots, g_m)^T, \mbox{ where } g_k = \sum_{i=0}^n G_{ki}f_i \mbox{ for all } 1\leq k \leq m. 
$$
This problem was considered by Fisk \cite[Chapter 3]{pbFiskB}, who proved some preliminary results. Since this appoach has been proved succesful in combinatorial situations, see \cite{pbSaVi} where it was used to prove e.g. that the type $D$ Eulerian polynomials are real--rooted, we take the opportunity to  give a complete characterization for the case of nonnegative polynomials. 
%

\begin{lemma}\label{pbfg}
If $(f_i)_{i=1}^n$ and $(g_i)_{i=1}^n$ are two  interlacing sequences of polynomials, then the polynomial 
$$
f_1g_n+ f_2g_{n-1}+ \cdots + f_ng_1
$$
is real--rooted.
\end{lemma}

\begin{proof}
By Theorem \ref{pbCS} it suffices to prove that the sequence $(f_ig_{n+1-i})_{i=1}^n$ is $2$-compatible. If $i<j$, then $f_ig_{n+1-j}$ is a common interleaver of 
$f_ig_{n+1-i}$ and $f_jg_{n+1-j}$. Hence the lemma follows from Theorem \ref{pbCS}. 
 \end{proof}

See \cite{pbSaVi} for a proof of the following lemma. 
\begin{lemma}\label{pblm}
Let $f$ and $g$ be two polynomials with nonnegative coefficients. Then $f \ll g$ if and only if for all $\lambda, \mu >0$, the polynomial 
$$
(\lambda x + \mu) f + g
$$
is real--rooted. 
\end{lemma}

\begin{theorem}\label{pbinchar}
Let  $G= (G_{ij}(x))$  be an $m \times n$ matrix of polynomials. Then $G : \mathcal{F}_n^+ \rightarrow \mathcal{F}_m^+$  if and only if 
\begin{enumerate}
\item $G_{ij}$ has nonnegative coefficients for all $i \in [m]$ and $j \in [n]$, and 
\item for all $\lambda, \mu >0$, $1\leq i< j \leq n$ and $1\leq k< \ell \leq m$ 
\begin{equation}\label{pbxvw}
(\lambda x + \mu) G_{kj}(x) + G_{\ell j}(x) \ll (\lambda x + \mu) G_{ki}(x) + G_{\ell i}(x).
\end{equation}
\end{enumerate}
\end{theorem}

\begin{proof}
Let 
$$
g_k = \sum_{i=0}^n G_{ki}f_i. 
$$
By Lemma \ref{pblm}, $G : \mathcal{F}_n^+ \rightarrow \mathcal{F}_m^+$  if and only if  for all $k < \ell$ and $\lambda, \mu >0$
$$
(\lambda x + \mu)g_k + g_\ell = \sum_{i=0}^n (  (\lambda x + \mu)G_{ki} + G_{\ell i})f_i =: \sum_{i=0}^n h_{n+1-i}f_i
$$
is real--rooted and has nonnegative coefficients. The sufficiency follows from Lemma \ref{pbfg}, since if \eqref{pbxvw} holds, then the sequence $(h_i)_{i=1}^n$ is interlacing. To prove the necessity, let $i<j$ and  $(f_r)_{r=1}^n$ be the interlacing sequence defined by 
$$
f_r(x) = 
\begin{cases} 1 & \mbox{ if } r=i,\\
\alpha x + \beta & \mbox{ if } r=j, \mbox{ and }\\
0 & \mbox{ otherwise}, 
\end{cases}
$$
where $\alpha, \beta >0$. 
Hence if $G : \mathcal{F}_n^+ \rightarrow \mathcal{F}_m^+$, then $h_{n+1-i} + (\alpha x + \beta)h_{n+1-j}$ is real--rooted for all $\alpha, \beta>0$. Thus $h_{n+1-j} \ll h_{n+1-i}$, by Lemma \ref{pblm}, which is \eqref{pbxvw}. \\
\mbox{ }  \end{proof}

\begin{corollary}
Let  $G= (G_{ij})$  be an $m \times n$ matrix over $\mathbb{R}$. Then $G : \mathcal{F}_n^+ \rightarrow \mathcal{F}_m^+$  if and only if $G$ is ${\rm TP}_2$, i.e.,  all minors of $G$ of size less than three are nonnegative. 
\end{corollary}

\begin{proof}
By Theorem \ref{pbinchar} we may assume that all entries of $G$ are nonnegative. Now 
$$
(\lambda x + \mu) G_{kj} + G_{\ell j} \ll (\lambda x + \mu) G_{ki} + G_{\ell i}
$$
for all $\lambda, \mu > 0$ 
if and only if 
$$
xG_{kj} + G_{\ell j} \ll x G_{ki} + G_{\ell i},
$$
which is seen to hold if and only if $G_{ki}G_{\ell j} \geq G_{\ell i}G_{k j}$. 
 \end{proof}

If $\lambda = (\lambda_1, \lambda_2, \ldots, \lambda_m)$ are integers such that $0 \leq \lambda_1 \leq \lambda_2 \leq \cdots \leq \lambda_m \leq n$, let $G_\lambda = (g_{ij}^\lambda(x))$ be the $m \times n$ matrix with entries 
$$
g_{ij}^\lambda(x) = 
\begin{cases}
x &\mbox{ if } 1 \leq j \leq \lambda_i \mbox{ and }  \\
1 &\mbox{ otherwise.}
\end{cases}
$$
The following corollary was first proved in \cite{pbSaVi}. 
\begin{corollary}\label{pbgla}
If $\lambda = (\lambda_1, \lambda_2, \ldots, \lambda_m)$ are integers such that $0 \leq \lambda_1 \leq \lambda_2 \leq \cdots \leq \lambda_m \leq n$, then 
$G_\lambda : \mathcal{F}_n^+ \rightarrow \mathcal{F}_m^+$. 
\end{corollary}
\begin{proof}
The possible $2 \times 2$ sub-matrices of $G_\lambda$ are 
$$
\left( \begin{array}{ c c }
x & x\\
x & x \\
\end{array}\right), 
\left( \begin{array}{ c c }
x & 1\\
x & x \\
\end{array}\right), 
\left( \begin{array}{ c c }
x & 1\\
x & 1 \\
\end{array}\right), 
\left( \begin{array}{ c c }
1 & 1\\
x & x \\
\end{array}\right), 
\left( \begin{array}{ c c }
1 & 1\\
x & 1 \\
\end{array}\right) \mbox{ and } 
\left( \begin{array}{ c c }
1 & 1\\
1 & 1 \\
\end{array}\right). 
$$
By Theorem  \ref{pbinchar} we need to check \eqref{pbxvw} for these matrices. For example for the second matrix from the right we need to check 
$$
(\lambda+1)x +\mu \ll x(\lambda x + \mu+1),
$$
for all $\lambda, \mu >0$, which is equivalent to 
$$
-\frac{\mu+1}{\lambda} \leq -\frac{\mu}{\lambda+1},
$$
which is certainly true. The other cases follows similarly. 
 \end{proof}
\begin{example}\label{pbaint}
Let $n$ be a positive integer and define polynomials $A_{n,i}(x)$, $i \in [n]$, by 
$$
A_{n,i}(x)= \sum_{{\sigma \in \mathfrak{S}_n} \atop \sigma(1)= i} x^{{\rm des}(\pi)}.
$$
By  conditioning on $\sigma(2) = k$, where $\sigma \in \mathfrak{S}_n$ and $\sigma(1)=i$, we see that 
$$
A_{n+1,i}(x) = \sum_{k<i} x A_{n,k}(x)+ \sum_{k \geq i} A_{n,k}(x), \ \ \ 1 \leq i \leq n+1.
$$
Hence if $\mathcal{A}_n = (A_{n,i}(x))_{i=1}^n$, then
$$
\mathcal{A}_{n+1} = G_{(0,1,2,\ldots, n)} \cdot \mathcal{A}_n.
$$
Since $\mathcal{A}_2=(1,x)$, we have by induction and Corollary \ref{pbgla} that $\mathcal{A}_n$ is an interlacing sequence of polynomials for all $n \geq 2$. 
\end{example}

\subsection{$\mathbf{s}$-Eulerian polynomials}
Corollary \ref{pbgla} was used by Savage and Visontai \cite{pbSaVi} to prove real--rootedness of a large family of $h^*$-polynomials. Let $\mathbf{s}=\{s_i\}_{i=1}^n$ be a sequence of positive integers. Define an integral polytope $P_\mathbf{s}$ by 
$$
P_\mathbf{s} = \left\{ (x_1,\ldots, x_n) \in \mathbb{R}^n : 0 \leq \frac {x_1}{s_1} \leq  \frac {x_2}{s_2} \leq \cdots \leq  \frac {x_n}{s_n} \leq 1 \right\}. 
$$
The $\mathbf{s}$--\emph{Eulerian polynomial} may defined as the $h^*$-polynomial of $P_\mathbf{s}$: 
$$
\sum_{k=0}^\infty i(P_\mathbf{s}, k) x^k = \frac {E_\mathbf{s}(x)}{(1-x)^{n+1}}.
$$
Savage and Schuster \cite{pbSaSc} provided a combinatorial description of $\mathbf{s}$-Eulerian polynomials. 
The $\mathbf{s}$-\emph{inversion sequences} are defined by 
$$
\mathcal{I}_\mathbf{s} = \{ \mathcal{E}=(e_1,\ldots, e_n) \in \mathbb{N}^n : e_i/s_i <1 \mbox{ for all } 1 \leq i \leq n\}.
$$
The \emph{ascent statistic} on $\mathcal{I}_\mathbf{s}$ is defined as 
$$
{\rm asc} (\mathcal{E}) = | \{ i \in [n]: e_{i-1}/s_{i-1} < e_{i}/s_{i} \}|,
$$
where $\mathcal{E}=(e_1,\ldots, e_n)$, $e_0=0$ and $s_0=1$. 
\begin{theorem}[\cite{pbSaSc}]
$$
E_\mathbf{s}(x)= \sum_{\mathcal{E} \in \mathcal{I}_\mathbf{s}} x^{{\rm asc} (\mathcal{E})}. 
$$
\end{theorem}
It turns out that several much studied families of polynomials in combinatorics are $\mathbf{s}$--Eulerian polynomials for various $\mathbf{s}$.  For example the $n$th ordinary Eulerian polynomial corresponds to $\mathbf{s}=(1,2,\ldots, n)$, while the $n$th Eulerian polynomial of type $B$ corresponds to $\mathbf{s}=(2,4, \ldots, 2n)$. If $\mathbf{s} =(s_1,\ldots, s_n)$, let 
$$
E_{\mathbf{s},i}(x)= \sum_{ {\mathcal{E} \in \mathcal{I}_\mathbf{s}} \atop {e_n =i} } x^{{\rm asc} (\mathcal{E})}. 
$$
It is not hard to see that the polynomials $E_{\mathbf{s},i}(x)$ satisfy the following recurrences which make them ideal for an application of Corollary \ref{pbgla}. 
\begin{lemma}[\cite{pbSaVi}]
If $\mathbf{s}=(s_1,\ldots, s_n)$, $n>1$, is a sequence of positive integers and $0 \leq i <n$, then 
$$
E_{\mathbf{s},i}(x) = \sum_{j=0}^{t_i-1} x E_{\mathbf{s}',j}(x) + \sum_{j=t_i}^{s_{n-1}-1}  E_{\mathbf{s}',j}(x),
$$
where $\mathbf{s}'= (s_1,\ldots, s_{n-1})$ and $t_i = \lceil i s_{n-1}/s_n \rceil$. 
\end{lemma}
An application of Corollary \ref{pbgla} proves the following theorem. 
\begin{theorem}[\cite{pbSaVi}]\label{pbSV}
If $\mathbf{s}=(s_1,\ldots, s_n)$ is a sequence of positive integers, then $E_\mathbf{s}(x)$ is real--rooted. Moreover if $n>1$, then the sequence $\{E_{\mathbf{s},i}(x) \}_{i=0}^{s_n-1}$ is interlacing. 
\end{theorem}

\subsection{Eulerian polynomials for finite Coxeter groups}
For undefined terminology on Coxeter groups we refer to \cite{pbBjBr}. Let $(W,S)$ be a Coxeter system. The \emph{length} of an element $w \in W$ is the smallest number $k$ such that 
$$
w=s_1s_2\cdots s_k,  \ \ \ \mbox{ where } s_i \in S \mbox{ for all } 1 \leq i \leq n. 
$$
Let $\ell_W(w)$ denote the length of $w$. The (right) \emph{descent set} of $w$ is 
$$
D_W(w)= \{ s \in S : \ell_W(ws) < \ell_W(w)\}, 
$$
and the \emph{descent number} is ${\rm des}_W(w) = |D_W(w)|$. The $W$--\emph{Eulerian polynomial} of a finite Coxeter group $W$ is the polynomial
$$
\sum_{w \in W} x^{{\rm des}_W(w)}
$$
which is known to be the $h$-polynomial of the Coxeter complex associated to $W$, see \cite{pbBreEu}. The type $A$ Eulerian polynomials are the common Eulerian polynomials. In \cite{pbBreEu}, Brenti conjectured that the Eulerian polynomial of any finite Coxeter group is real--rooted. Brenti's conjecture is true for type $A$ and $B$ Coxeter groups \cite{pbBreEu, pbFrob}, and one may check with the aid of the computer that the conjecture holds for the exceptional groups $H_3, H_4, F_4, E_6, E_7$, and $E_8$. Moreover, the Eulerian polynomial of the 
direct product of two finite Coxeter groups is the product of the Eulerian polynomials of the two groups. Hence it remains to prove Brenti's conjecture for type $D$ Coxeter groups. The type $D$ case resisted many attempts, and it was not until very recently that the first sound proof was given by Savage and Visontai \cite{pbSaVi}. Their proof used compatibility arguments and ascent sequences.  We will give a similar proof below that avoids the detour via ascent sequences.

Recall that a combinatorial description of a rank $n$ Coxeter group of type $B$ is the group $B_n$ of signed permutations $\sigma : [\pm n] \rightarrow [\pm n]$, where $[\pm n]= \{\pm 1,\ldots, \pm n\}$,  such that $\sigma(-i) = -\sigma(i)$ for all $i \in [\pm n]$. An element $\sigma \in B_n$ is conveniently encoded by the \emph{window notation} as a word $\sigma_1\cdots \sigma_n$, where $\sigma_i=\sigma(i)$. The type $B$ \emph{descent number} of $\sigma$ is then 
$$
{\rm des}_B(\sigma)= |\{i \in [n] : \sigma_{i-1} > \sigma_i\}|,
$$
where $\sigma_0 :=0$, see \cite{pbBreEu}. The $n$th \emph{type $B$ Eulerian polynomial} is thus 
$$
B_n(x)= \sum_{\sigma \in B_n} x^{{\rm des}_B(\sigma)}.
$$

A combinatorial description of a rank $n$ Coxeter group of type $D$ is the group $D_n$ consisting of all elements of $B_n$ with an even number of negative entries in their window notation. The type $D$ \emph{descent number} of $\sigma \in D_n$ is then 
$$
{\rm des}_D(\sigma)= |\{i \in [n] : \sigma_{i-1} > \sigma_i\}|,
$$
where $\sigma_0 := -\sigma_2$, see \cite{pbBreEu}. The $n$th \emph{type $D$ Eulerian polynomial} is 
$$
D_n(x)= \sum_{\sigma \in D_n} x^{{\rm des}_D(\sigma)}.
$$
For $n\geq 2$ and $k \in [\pm n]$, let 
$$
D_{n,k}(x) =  \sum_{ {\sigma \in D_n} \atop {\sigma_n=-k}} x^{{\rm des}_D(\sigma)}.
$$
If $k \notin [\pm n]$, we set $D_{n,k}(x) := 0$. 
The following table is conveniently generated by the recursion in Lemma \ref{pbDRe} below. 

\begin{equation}\label{pbdarray}
\begin{array}{r | l l l }
k & D_{2,k} & D_{3,k} & D_{4,k} \\ \hline 
-4 & 0 &  0&  (x+1)(x^2+10x+1)   \\
-3 & 0 & (x+1)^2&   2x(x+1)(x+5)  \\
 -2 & 1 & x(3+x)& x(3x^2+14x+7) \\
 -1 & x &2x(x+1) &  x(5x^2+14x+5) \\
 1 & x & 2x(x+1)&   x(5x^2+14x+5)  \\ 
2 &  x^2 &x(3x+1) &   x(7x^2+14x+3) \\ 
3 &  0 & x(x+1)^2 & 2x(x+1)(5x+1)  \\
4 & 0 & 0 &   x(x+1)(x^2+10x+1) \\
\end{array}
\end{equation}

Note that the type $D$ descents make sense for any element of $B_n$, where $n \geq 2$. 

\begin{lemma}
If $n\geq 2$, then 
\begin{equation}\label{pbdb}
D_{n,k}(x) = \frac 1 2 \sum_{ {\sigma \in B_n} \atop {\sigma_n=-k}} x^{{\rm des}_D(\sigma)}.
\end{equation}
\end{lemma}

\begin{proof}
For $k \in [n]$, let $\phi_k : B_n \rightarrow B_n$ be the involution that swaps the letters $k$ and $-k$ in the window notation of the permutation. Clearly $\phi_1$ is a bijection between $D_n$ and $B_n \setminus D_n$  which preserves the type $D$ descents for all $n \geq 2$. This proves \eqref{pbdb} for $k \notin \{1,-1\}$. 

For $k \in [\pm n]$, let  $B_n[k]$ be the set of $\sigma \in B_n$ with $\sigma_n=k$. Then $\phi_1$ is a bijection between $B_n[1]$ and $B_n[-1]$ which preserves the type $D$ descents for all $n \geq 2$. Similarly let  $D_n[k]$ be the set of $\sigma \in D_n$ with $\sigma_n=k $. Now 
$B_n[1]=D_n[1] \cup \phi_1(D_n[-1])$ and $B_n[-1]=D_n[-1] \cup \phi_1(D_n[1])$, where the unions are disjoint. Hence to prove \eqref{pbdb} for $k=\pm 1$,  it remains to prove $D_{n,1}(x)= D_{n,-1}(x)$. We prove this by induction on $n \geq 2$, where the case $n=2$ is easily checked. 

Consider the involution $\phi_2 \phi_1 : D_n[1] \rightarrow D_n[-1]$, where $n \geq 3$. Then $\phi_2 \phi_1$ preserves type $D$ descents on $\sigma$ unless $\sigma_{n-1}= \pm 2$. Hence it remains to prove that the type $D$ descent generating polynomials of 
$D_n[2,1] \cup D_n[-2,1]$ and $D_n[2,-1] \cup D_n[-2,-1]$ agree, where $D_n[k,\ell]$ is the set of $\sigma \in D_n$ such that $\sigma_{n-1}=k$ and $\sigma_n=\ell$. By induction we have 
$$
\begin{array}{ c | l }
\mbox{Set} & \mbox{Generating polynomial of set}\\ \hline
D_n(2,1) & xD_{n-1,1}(x) \\
D_n(-2,1) & D_{n-1,1}(x) \\
D_n(2,-1) & xD_{n-1,1}(x) \\
D_n(-2,-1) & xD_{n-1,1}(x) \\
\end{array},
$$
and the lemma follows. 
 \end{proof}

\begin{lemma}\label{pbDRe}
If  $n \geq 2$ and $i \in [\pm n]$, then 
\begin{align*}
D_{n+1,i}(x) &= \sum_{k \leq i} x D_{n,k}(x) + \sum_{k > i} D_{n,k}(x), \ \ \mbox{ if } i < 0 \mbox{ and }\\
D_{n+1,i}(x) &= \sum_{k < i} x D_{n,k}(x) + \sum_{k \geq i} D_{n,k}(x), \ \ \mbox{ if } i >0. 
\end{align*}
\end{lemma}

\begin{proof}
The lemma follows easily by using the alternative description \eqref{pbdb} of $D_{n,i}(x)$, and keeping track of $\sigma_n$, where $\sigma \in D_{n+1}[-i]$. We leave the details to the reader.
 \end{proof}

\begin{theorem}\label{pbmainD}
Let $n \geq 2$. The type $D$ Eulerian polynomial $D_n(x)$ is real--rooted. 

Moreover for each $k \in [\pm n]$, the polynomial $D_{n,k}(x)$ is real--rooted, and if $n \geq 4$, then the sequence $\mathcal{D}_n:=(D_{n,k}(x))_{k \in [\pm n]}$ is interlacing. 
\end{theorem}

\begin{proof}
One may easily check that $D_n(x)$ and $D_{n,k}(x)$ are real--rooted whenever $2\leq n \leq 4$ and $k \in [\pm n]$, see \eqref{pbdarray}. The sequence $\mathcal{D}_4$ is interlacing, see \eqref{pbdarray}. By Lemma \ref{pbDRe}, up to a relabeling of $[\pm n]$,
$$
\mathcal{D}_{n+1} = G_{\lambda^n} \mathcal{D}_n, 
$$
where $\lambda^n$ is a weakly increasing sequence. The matrix $G_{\lambda^n}$ is of the type appearing in Corollary \ref{pbgla}. Hence the theorem follows from Corollary \ref{pbgla}. 
 \end{proof}


\begin{theorem}[Frobenius \cite{pbFrob}, Brenti \cite{pbBreEu}, Savage--Visontai \cite{pbSaVi}]
The Eulerian polynomial of any finite Coxeter group is real--rooted. 
\end{theorem}

\begin{remark}
For $n \geq 1$ and $i \in [\pm n]$, define 
$$
B_{n,i}(x)=  \sum_{ {\sigma \in B_n} \atop {\sigma_n=-i}} x^{{\rm des}_B(\sigma)}.
$$
Then $B_{n,i}$ satisfies the same recursion as in Lemma \ref{pbDRe}, because the proof is ignorant to what happens in the far left  in the window notation of an element of $B_n$. Moreover this recursion is valid for all $n \geq 1$. Since $(B_{1,-1}(x), B_{1,1}(x)) = (1,x)$ is interlacing, induction and Corollary \ref{pbgla} implies that the sequence $(B_{n,i}(x))_{i \in [\pm n]}$ is an interlacing sequence of polynomials for all $n \geq 1$. 
\end{remark}

\section{Multivariate techniques}
To prove that a family of univariate polynomials are real--rooted it is sometimes easier to work with multivariate analogs of the polynomials. As alluded to in Section~\ref{pbSecPres}, a fruitful generalization of real--rootedness for multivariate polynomials is that of (real-) stable polynomials. There are several benefits in a multivariate approach; the proofs sometimes become more transparent, several powerful inequalities are available for multivariate stable polynomials, it may give you a better understanding for the combinatorial problem at hand. An important class of stable polynomials are \emph{determinantal} polynomials. 
\begin{proposition}\label{pbdetprop}
Let $A_1, \ldots, A_n$ be positive semidefinite hermitian matrices, and $A_0$ a hermitian matrix. Then the polynomial 
$$
P(x_1, \ldots, x_n) = \det( A_0 + x_1A_1+ \cdots+ x_n A_n) 
$$
is either stable or identically zero. 
\end{proposition}
\begin{proof}
By Hurwitz' theorem \cite[Footnote 3, p.~96]{pbCOSW} and a standard approximation argument we may assume that $A_1$ is positive definite. Let $\mathbf{x}= (a_1+ib_1, \ldots, a_n+ib_n) \in \mathbb{C}^n$ be such that $a_j \in \mathbb{R}$ and $b_j >0$ for all $1\leq j \leq n$. We need to prove that $P(\mathbf{x}) \neq 0$. Now 
$
P(\mathbf{x})= \det(iB-A)
$,
where $B= b_1A_1+\cdots+b_nA_n$ is positive definite and $A=-A_0-a_1A_1-\cdots-a_nA_n$ is hermitian. Hence $B$ has a square root and thus  
$P(\mathbf{x})= \det(B)\det(iI-B^{-1/2}AB^{-1/2})\neq 0$, where $I$ is the identity matrix, since $B^{-1/2}AB^{-1/2}$ is hermitian and thus has real eigenvalues only. 
 \end{proof}
For $n=2$ a converse of Proposition \ref{pbdetprop} holds, see Theorem \ref{pbHV-thm}. The analog of Theorem \ref{pbHV-thm} for $n \geq 3$ fails to be true by a simple count of parameters. For possible  partial converses of Proposition \ref{pbdetprop}, see the survey \cite{pbVinS}.

Recently attempts have been made to find appropriate multivariate analogs of frequently studied real--rooted univariate polynomials in combinatorics. Let us illustrate by describing a multivariate 
Eulerian polynomial. For $\sigma \in \mathfrak{S}_n$ let 
\begin{align*}
{\rm DB}(\sigma) &= \{\sigma(i) : \sigma(i-1)>\sigma(i)\}, \mbox{ and }\\
{\rm AB}(\sigma) &= \{\sigma(i) : \sigma(i-1)<\sigma(i)\},
\end{align*}
where $\sigma(0)=\sigma(n+1) =\infty$, be the set of \emph{descent bottoms} and \emph{ascent bottoms} of $\sigma$, respectively. Let $A_n(\mathbf{x},\mathbf{y})$ be the polynomial in $\mathbb{R}[x_1,\ldots, x_n, y_1, \ldots, y_n]$ defined by 
$$
A_n(\mathbf{x},\mathbf{y}) = \sum_{\sigma \in \mathfrak{S}_n} w(\sigma), \ \ \mbox{ where } w(\sigma)= \prod_{i \in {\rm DB}(\sigma)}\!\!\! \!\!x_i \!\prod_{j \in {\rm AB}(\sigma)}\!\!\! \!\! y_i .
$$
For example $w(573148926)= x_5x_3x_1x_2 y_5y_1y_4y_8y_2y_6$. 
Generate a permutation $\sigma'$ in $\mathfrak{S}_{n}$ by inserting the letter $1$ in a slot between two adjacent letters in a permutation $\sigma_0 \sigma_1  \cdots \sigma_{n-1} \sigma_{n}$ of $\{2,3,\ldots,n\}$ (where $\sigma_0=\sigma_{n} =\infty$). Note that there is an obvious one--to--one correspondence between the slots and the variables appearing in $w(\sigma')$. Thus if we insert $1$ in the slot corresponding to the variable $z$, then 
$$
w(\sigma) = x_1y_1\frac \partial {\partial z} w(\sigma')
$$
since the descent/ascent bottom corresponding to $z$ in $\sigma'$ will be destroyed, and $1$ becomes an ascent- and descent bottom. We have proved 
$$
A_{n}(\mathbf{x}, \mathbf{y}) = x_1y_1 \left(\sum_{j=2}^{n} \frac \partial {\partial x_j} + \frac \partial {\partial y_j}\right)A_{n-1}(\mathbf{x}^*,\mathbf{y}^*),
$$
where $\mathbf{x}^*=(x_2,\ldots, x_n)$ and $\mathbf{y}^*=(y_2,\ldots, y_n)$. To prove that $A_{n}(\mathbf{x}, \mathbf{y})$ is stable for all $n$ it remains to prove that the operators of the form $\sum_{j=1}^{n}  \partial/{\partial x_j}$ preserve stability. Stability preservers were recently characterized in \cite{pbInv}. The following theorem is the algebraic characterization. For 
$\kappa \in \mathbb{N}^n$, let $\mathbb{C}_\kappa[x_1,\ldots, x_n]$ be the linear space of all polynomials that have degree at most $\kappa_i$ in $x_i$ for each $1 \leq i \leq n$. The \emph{symbol} of a linear operator 
$T : \mathbb{C}_\kappa[x_1, \ldots, x_n] \rightarrow \mathbb{C}[x_1,\ldots, x_m]$ is the polynomial 
$$
G_T(x_1,\ldots, x_m, y_1,\ldots, y_n) = T\left( (x_1+y_1)^{\kappa_1} \cdots (x_n+y_n)^{\kappa_n}\right),
$$
where $T$ acts on the $y$-variables as if they were constants.  
\begin{theorem}[\cite{pbInv}]\label{pbmultistab}
Let $T : \mathbb{C}_\kappa[x_1, \ldots, x_n] \rightarrow \mathbb{C}[x_1,\ldots, x_m]$ be a linear operator of rank greater than one. Then 
$T$ preserves stability if and only if $G_T$ is stable. 
\end{theorem}
The symbol of the operator $T=\sum_{j=1}^{n}  \partial/{\partial x_j}$ is 
$$
G_T= (x_1+y_1)^{\kappa_1}\cdots (x_n+y_n)^{\kappa_n} \sum_{j=1}^n \frac {\kappa_j} {x_j+y_j}.
$$
Hence if ${\rm Im} (x_j) >0$ and ${\rm Im} (y_j) >0$ for all $1 \leq j \leq n$, then ${\rm Im} (x_j+y_j)^{-1} <0$, and hence the symbol is non-zero. Thus $G_T$ is stable and by induction and Theorem \ref{pbmultistab}, $A_n(\mathbf{x},\mathbf{y})$ is stable for all $n \geq 1$. 

The multivariate Eulerian polynomials above and more general Eulerian-like polynomials were introduced in \cite{pbMCPC} and used to prove the Monotone Column Permanent Conjecture of Haglund, Ono and Wagner \cite{pbHOW}. Suppose $A=(a_{ij})_{i,j=1}^n$ is a real matrix which is weakly increasing down columns. Then the Monotone Column Permanent Conjecture stated that the permanent of the matrix $(a_{ij}+x)_{i,j=1}^n$, where $x$ is a variable, is real--rooted. Subsequently multivariate Eulerian polynomials for colored permutations and various other models have been studied \cite{pbBrLeVi,pbHaVi,pbViWi,pbCHY}. 

\subsection{Stable polynomials and matroids}\label{pbSecHPP}
Let $E$ be a finite set and let $\mathbf{x}= (x_e)_{e \in E}$ be independent variables. The \emph{support} of a multiaffine polynomial 
$$
P(\mathbf{x})= \sum_{S \subseteq E} a(S) \prod_{e \in S} x_e,
$$
is the set system ${\rm Supp}(P)= \{ S \subseteq E: a(S) \neq 0\}$. Choe, Oxley, Sokal and Wagner \cite{pbCOSW} proved the following striking relationship between stable polynomials and matroids. 
\begin{theorem}\label{pbCOSWthm}
The support of a homogeneous, multiaffine and stable polynomial is the set of bases of a matroid. 
\end{theorem}
Hence Theorem \ref{pbCOSWthm} suggests an alternative way of representing matroids. A matroid $M$, with set of bases $\mathcal{B}$, has the \emph{half-plane property} (HPP) if its \emph{bases generating polynomial}
$$
P_M(\mathbf{x})= \sum_{B \in \mathcal{B}} \prod_{e \in B} x_e
$$
is stable, and $M$ has the \emph{weak half-plane property} (WHPP) if there are positive numbers $a(B)$, $B \in \mathcal{B}$, such that 
$$
\sum_{B \in \mathcal{B}} a(B) \prod_{e \in B} x_e
$$
is stable. For example, the Fano matroid $F_7$ is not WHPP, see \cite{pbHPP}. The fact that graphic matroids are HPP is a consequence of the \emph{Matrix--tree theorem} and Proposition \ref{pbdetprop}. Suppose $V=[n]$, and let $\{\delta_i\}_{i=1}^n$ be the standard basis of $\mathbb{R}^n$. The \emph{weighted Laplacian} of a connected graph 
$G=(V,E)$ is defined as 
$$
L_G(\mathbf{x})= \sum_{e\in E}x_e(\delta_{e_1}-\delta_{e_2})(\delta_{e_1}-\delta_{e_2})^T,
$$
where $e_1$ and $e_2$ are the vertices incident to $e\in E$. We refer to \cite[Theorem VI.29]{pbTutte} for a proof of the next classical theorem that goes back to Kirchhoff and Maxwell. Let $T_G(\mathbf{x})$ be the \emph{spanning tree polynomial} of $G$, i.e., the bases generating polynomial of the graphical matroid associated to $G$. 
\begin{theorem}[Matrix--tree theorem]\label{pbmatrix-tree}
For $i \in V$, let $L_G(\mathbf{x})_{ii}$ be the matrix obtained by deleting the column and row indexed by $i$ in $L_G(\mathbf{x})$.  Then 
$$
T_G(\mathbf{x}) = \det(L_G(\mathbf{x})_{ii}).
$$
\end{theorem}
Clearly the matrices in the pencil $L_G(\mathbf{x})_{ii}$ are positive semidefinite. Hence that graphic matroids are HPP follows from Theorem \ref{pbmatrix-tree} and 
Proposition \ref{pbdetprop}. A similar reasoning proves that all regular matroids are HPP, and that all matroids representable over $\mathbb{C}$ are WHPP, see \cite{pbHPP,pbCOSW}. On the other hand, the V\'amos cube $V_8$ is not representable over any field, and still $V_8$ is HPP \cite{pbWaWe}. For further results on the relationship between stable polynomials and matroids we refer to \cite{pbHPP, pbBrGo,pbCOSW,pbWaWe}. 

\subsection{Strong Rayleigh measures}
Stability implies  several strong inequalities among the coefficients. Note that the multivariate Eulerian polynomial above is \emph{multiaffine}, i.e., it is of degree at most one in each variable. We may view multiaffine polynomials with nonnegative coefficients as discrete probability measures. If $E$ is a finite set, $\mathbf{x}=(x_e)_{e\in E}$ are independent variables, and 
$$
P(\mathbf{x}) = \sum_{S \subseteq E} a(S) \prod_{e \in S}x_e ,
$$
is a multiaffine polynomial with nonnegative coefficients normalized so that $P(1,\ldots1)=1$, we may define a discrete probability measure $\mu$ on $2^E$ by setting $\mu(S)=a(S)$ for each $S \in 2^E$. Then $P_\mu:=P$ is the multivariate \emph{partition function} of $\mu$. A discrete probability measure $\mu$ is called \emph{strong Rayleigh} if $P_\mu$ is stable. Hence the measure $\mu_n$ on $2^{[2n]}$,  defined by  
$$
\mu_n(S)= \frac 1 {n!} |\{ \sigma \in \mathfrak{S}_n :  {\rm DB}(\sigma) \cup \{ i +n  : i \in {\rm AB}(\sigma) \}  = S  \}|
$$
is strong Rayleigh. A fundamental strong Rayleigh measure is the \emph{uniform spanning tree measure}, $\mu_G$, associated to a connected graph $G=(V,E)$. This is the measure on $2^E$ defined by 
$$
\mu_G(S) = \frac 1 t\begin{cases}
1 &\mbox{ if } S \mbox{ is a spanning tree}, \\
0 &\mbox{ otherwise} 
\end{cases}, 
$$
 where $t$ is the number of spanning trees of $G$. The uniform spanning tree measures --- and more generally the uniform measure on the set of bases of any HPP matroid --- is strong Rayleigh by the discussion in Section \ref{pbSecHPP}.  
 
 A general class of strong Rayleigh measures containing the uniform spanning tree measures is the class of \emph{determinantal measures}, see \cite{pbLyons}. Let $C$ be a hermitian $n \times n$ contraction matrix, i.e., a positive semidefinite matrix with all its eigenvalues located in the interval $[0,1]$. Define a probability measure on $2^{[n]}$ by 
 $$\mu_C(\{T : T \supseteq S\}) =\det C(S), \ \ \ \mbox{ for all } S \subseteq [n], 
 $$ 
 where $C(S)$ is the submatrix of $C$ with rows and columns indexed by $S$.   
Using Proposition \ref{pbdetprop}, it is not hard to prove that $\mu_C$ is strong Rayleigh, see \cite{pbJAMS}. 

Negative dependence is an important notion in probability theory, statistics and statistical mechanics, see the survey 
\cite{pbPemSu2}. 
 In \cite{pbJAMS} several strong negative dependence properties of strong Rayleigh measures were established. Identify $2^E$ with $\{0,1\}^E$. 
A probability measure $\mu$ on $\{0,1\}^n$  is \emph{negatively associated} if 
$$
\int fg d\mu \leq \int f d\mu \int gd\mu,
$$
whenever $f,g:\{0,1\}^n \to \mathbb{R}$ are increasing functions depending on disjoint sets of variables, i.e., $f(\eta)$ only depends on the variables $\eta_i, i \in A$, and $g(\eta)$ only depends on the variables $\eta_i, i \in B$, where $A\cap B =\emptyset$. In particular setting $f(\eta)= \eta_i$ and $g(\eta)= \eta_j$, where $i \neq j$, we see that $\mu$ is \emph{pairwise negatively correlated} i.e., 
$$
\mu(\eta : \eta_i=\eta_j=1) \leq \mu(\eta : \eta_i=1)\mu(\eta : \eta_j=1).
$$
\begin{example}
For $n=2$, a discrete probability measure $\mu$ defined by $\mu(\emptyset)=a, \mu(\{1\})=b, \mu(\{2\})=c,\mu(\{1\})=d$, with $a+b+c+d=1$ is pairwise negatively correlated if and only if $d(a+b+c+d)\leq (b+d)(c+d)$, i.e., 
 if and only if $ad\leq bc$. Also, it is easy to see that a real polynomial $a+bx_1+cx_2+dx_1x_2$ is stable if and only if $ad \leq bc$. By the next theorem the notions strong Rayleigh, negative association and pairwise negative correlation agree for $n=2$. 
 \end{example}

\begin{theorem}[\cite{pbJAMS}]
If $\mu$ is a discrete probability measure which is strong Rayleigh, then it is negatively associated.
\end{theorem}

Recently Pemantle and Peres \cite{pbPePe} proved general concentration inequalities for strong Rayleigh measures. A function 
$f : \{0,1\}^n \rightarrow \mathbb{R}$ is \emph{Lipschitz-1} if 
$$
|f(\eta)-f(\xi)| \leq d(\eta,\xi), \ \ \mbox{ for all } \eta, \xi \in \{0,1\}^n, 
$$
where $d$ is the \emph{Hamming distance}, $d(\eta,\xi)= |\{i \in [n]: \eta_i \neq \xi_i\}|$.  
\begin{theorem}[Pemantle and Peres, \cite{pbPePe}]
Suppose $\mu$ is a probability measure on $\{0,1\}^n$ whose partition function is stable and has mean $m = \mathbb{E}(\sum_{i=1}^n\eta_i)$. If $f$ is any Lipschitz-1 function on $\{0,1\}^n$, then  
$$
\mu( \eta : |f(\eta)-\mathbb{E}f|>a) \leq 5 \exp\left( - \frac {a^2}{16(a+2m)}\right).
$$
\end{theorem}

\subsection{The symmetric exclusion process}
The \emph{symmetric exclusion process} (with creation and annihilation) is a Markov process  that models particles  jumping on a countable set of sites. Here we will just consider the case when we have a finite set of sites $[n]$. Given a \emph{symmetric} matrix $Q= (q_{ij})_{i,j=1}^n$  of nonnegative numbers and vectors $b=(b_i)_{i=1}^n$ and $d=(d_i)_{i=1}^n$ of nonnegative numbers, define a continuous time Markov process on $\{0,1\}^n$ as follows. Let $\eta \in \{0,1\}^n$ represent the configuration of the particles, with $\eta_i=1$ meaning that site $i$ is occupied, and $\eta_i=0$ that site $i$ is vacant. Particles at occupied sites jump to vacant sites at specified rates. More precisely, these are the transitions in the Markov process, which we denote by ${\rm SEP}(Q,b,d)$, see Fig.~\ref{pbtransi}:
\begin{itemize}
\item[(J)] A particle jumps from site $i$ to site $j$ at rate $q_{ij}$: The configuration $\eta$ is unchanged unless  $\eta_i=1$ and $\eta_j=0$, and then $\eta_i$ and $\eta_j$ are exchanged in $\eta$. 
\item[(B)] A particle at site $i$ is created (is born) at rate $b_i$: The configuration $\eta$ is unchanged unless  $\eta_i=0$, and then $\eta_i$ is changed from a zero to a one in $\eta$. 
\item[(D)] A particle at site $i$ is annihilated (dies) at rate $d_i$: The configuration $\eta$ is unchanged unless  $\eta_i=1$, and then $\eta_i$ is changed from a one to a zero in $\eta$. 
\end{itemize}
\begin{figure}
\begin{center}
 \includegraphics[width=5cm]{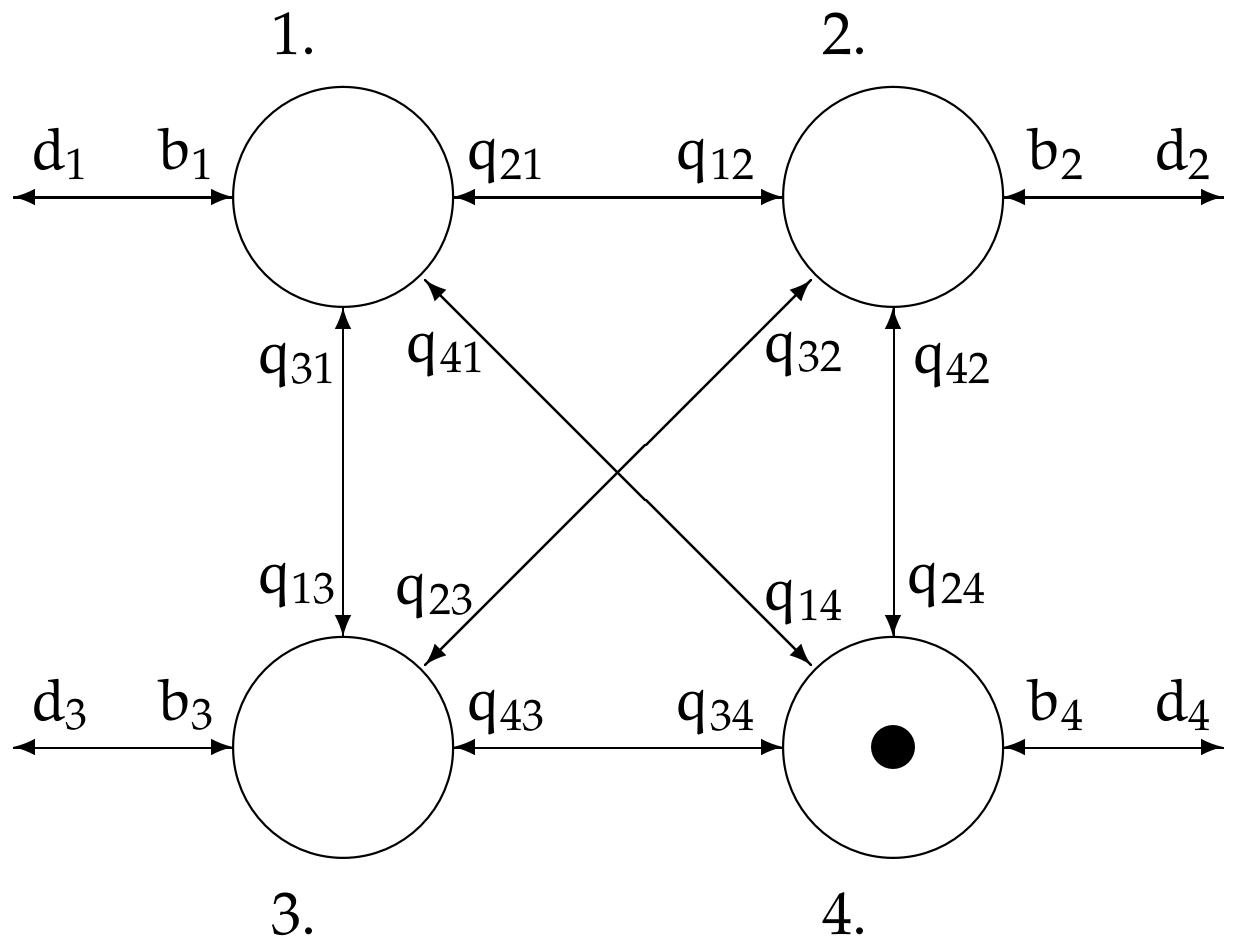}
\end{center}
\caption{The transitions in ${\rm SEP}(Q,b,d)$ on $4$ sites, where $q_{ij}=q_{ji}$. }\label{pbtransi}
\end{figure} 

It was proved in \cite{pbJAMS,pbWaS} that ${\rm SEP}(Q,b,d)$ preserves the family of strong Rayleigh measures. 
\begin{theorem}\label{pbSEPs}
If the initial distribution of a symmetric exclusion process ${\rm SEP}(Q,b,d)$ is strong Rayleigh, then the distribution is strong Rayleigh for all positive times. 
\end{theorem}
An immediate consequence of Theorem \ref{pbSEPs} is that the stationary distribution (if unique) of the symmetric exclusion process is strong Rayleigh. 
\begin{corollary}\label{pbSEPsC}
If a symmetric exclusion process ${\rm SEP}(Q,b,d)$ is irreducible and positive recurrent, then the unique stationary distribution is strong Rayleigh.
\end{corollary}
\begin{proof}
Choose an initial distribution which is strong Rayleigh. Then the partition function, $P_t(\mathbf{x})$, of the distribution at time $t$ is stable for all $t>0$, by Theorem \ref{pbSEPs}. The partition function of the stationary distribution is given by $\lim_{t \to \infty}P_t(\mathbf{x})$. By Hurwitz' theorem \cite[Footnote 3, p.~96]{pbCOSW} the partition function of the stationary distribution is stable, i.e., the stationary distribution is strong Rayleigh. 
 \end{proof}
In view of Corollary \ref{pbSEPsC} it would be interesting to find the stationary distributions of ${\rm SEP}(Q,b,d)$ for specific parameters $Q,b$, and $d$. This was achieved by Corteel and Williams \cite{pbCoWi} for the parameters
\begin{align}\label{pbparam}
q_{ij} &= \nonumber
\begin{cases}
1 &\mbox{ if } |j-i|=1  \mbox{ and }\\
0 &\mbox{ if } |j-i|>1. \\
\end{cases}, \\ 
b &= (\alpha, 0, \ldots, 0, \delta), \\ \nonumber
d &= (\gamma, 0, \ldots, 0, \beta). \\ \nonumber
\end{align}
Hence the particles jump on a line, where particles are only allowed to jump to neighboring sites, and be created and annihilated at the endpoints. The description of the stationary distribution is in terms of combinatorial objects called staircase tableaux. The special case when $\delta=\gamma=0$ is related to multivariate Eulerian polynomials. The \emph{excedence set}, $\mathcal{X}(\sigma) \subseteq [n]$,  of a signed permutation $\sigma \in B_n$ was defined by Steingr\'imsson  \cite{pbStein} as 
$$
i \in \mathcal{X}(\sigma)  \mbox{ if and only if } 
\begin{cases} 
|\sigma(i)|>i, \mbox{ or};\\
\sigma(i) = -i.
\end{cases}
$$
If $\sigma \in B_n$, let $|\sigma| \in \mathfrak{S}_n$ be the permutation where $i \mapsto |\sigma(i)|$ for all $1\leq i \leq n$. A cycle $c$ of $|\sigma|$ is called a \emph{negative cycle} of $\sigma \in B_n$ if $\sigma(j)<0$, where $|\sigma(j)|$ is the maximal element of $c$. Otherwise $c$ is called a \emph{positive cycle} of $\sigma$. Let $c_-(\sigma)$ and $c_+(\sigma)$ be the number of negative- and positive cycles of $\sigma$, respectively. 
\begin{theorem}[\cite{pbBrLeVi}]
The multivariate partition function of the symmetric exclusion process on $2^{[n]}$ with parameters as in \eqref{pbparam}, with $\delta=\gamma=0$, is a constant multiple of 
\begin{equation}\label{pbSEB}
\sum_{\sigma \in B_n} \left(\frac 2 \alpha\right)^{c_-(\sigma)} \left(\frac 2 \beta\right)^{c_+(\sigma)} \prod_{i \in  \mathcal{X}(\sigma)} x_i. 
\end{equation}
\end{theorem}
Note that by Corollary \ref{pbSEPsC}, the polynomial \eqref{pbSEB} is stable.

\begin{problem}
Find the stationary distribution of  ${\rm SEP}(Q,b,d)$ for parameters other than \eqref{pbparam}. 
\end{problem}

\subsection{The Grace--Walsh--Szeg\H{o} theorem, and the proof of Theorem~\protect\ref{pbite}}\label{pbSecGWS}
The proof of Theorem \ref{pbite} is an excellent example of how multivariate techniques may be used to prove statements about the zeros of univariate polynomials. The proof uses a combinatorial symmetric function identity and the Grace--Walsh--Szeg\H{o} theorem, which is undoubtedly one of the most useful theorems governing the location of zeros of polynomials, see \cite{pbRS}.

A {\em circular region} is a proper subset of the complex plane that is bounded by either a circle or a straight line, and is either open or closed.   
\begin{theorem}[Grace--Walsh--Szeg\H{o}]\label{pbGWS}
Let $f \in \mathbb{C}[z_1, \ldots, z_n]$ be a multiaffine and symmetric polynomial, and let $K$ be a circular region. Assume that either $K$ is convex or that the degree of $f$ is $n$. For any $\zeta_1, \ldots, \zeta_n \in K$ there is a number $\zeta \in K$ such that 
$
f(\zeta_1, \ldots, \zeta_n)= f(\zeta, \ldots, \zeta). 
$
\end{theorem}
The second ingredient in the proof of Theorem \ref{pbite} is the following symmetric function identity. Let 
$e_k(\mathbf{x})$ be the $k$th elementary symmetric function in the variables $\mathbf{x}=(x_1,\ldots, x_n)$.
\begin{lemma}\label{pbekid}
For  nonnegative integers $n$,  
\begin{equation}\label{pbbeauty}
\sum_{k=0}^n (e_k(\mathbf{x})^2-e_{k-1}(\mathbf{x})e_{k+1}(\mathbf{x}))= e_n(\mathbf{x})\sum_{k=0}^{\lfloor n/2 \rfloor}  C_k e_{n-2k}\left(\mathbf{x}+\frac 1 {\mathbf{x}}\right),
\end{equation}
$\mathbf{x}+1/\mathbf{x} = (x_1+1/x_1, \ldots, x_n+1/x_n)$ and $C_k=\binom {2k} k /(k+1)$, $k \in \mathbb{N}$, are the Catalan numbers.  
\end{lemma}
\begin{proof}
For undefined symmetric function terminology, we refer to \cite[Chapter 7]{pbStanEC2}. The polynomial $e_k(\mathbf{x})^2-e_{k-1}(\mathbf{x})e_{k+1}(\mathbf{x})$ is the Schur--function $s_{2^k}(\mathbf{x})$, where $2^k=(2,2,\ldots,2)$. We may rewrite \eqref{pbbeauty} as
\begin{equation}\label{pbcat-type}
\sum_{k=0}^n s_{2^k}(\mathbf{x}) = \sum_{k=0}^{\lfloor n/2 \rfloor} C_k \sum_{|S|=2k} \prod_{i \in S}x_i \prod_{j \notin S}(1+x_j^2). 
\end{equation}
 By the combinatorial definition of the Schur--function, the left hand side of \eqref{pbcat-type} is the generating polynomial of all semi--standard Young tableaux  with entries in $\{1, \ldots, n\}$, that are  of shape $2^k$ for some $k \in \mathbb{N}$. Call this set $\mathcal{A}_n$. Given $T \in \mathcal{A}_n$, let $S$ be the set of entries which occur only ones in $T$. By deleting the remaining entries we obtain a standard Young tableau  of shape $2^{k}$, where $2k=|S|$. There are exactly $C_k$ standard Young tableaux of shape 
$2^k$ with set of entries $S$, see e.g. \cite[Exercise 6.19.ww]{pbStanEC2}. It is not hard to see that the original semi--standard Young tableau is then determined by the set of duplicates. This explains the right hand side 
of \eqref{pbcat-type}.  
 \end{proof}

\begin{proof}[Proof of Theorem \ref{pbite}]
Let $P(x)= \sum_{k=0}^n a_kx^k= \prod_{k=0}^n(1+\rho_k x)$, where $\rho_k >0$ for all $1 \leq k \leq n$, and let 
$$
Q(x)= \sum_{k=0}^n (a_k^2-a_{k-1}a_{k+1})x^k.
$$
Suppose there is a number $\zeta \in \mathbb{C}$, with $\zeta \notin \{x \in \mathbb{R}: x \leq 0\}$, for which $Q(\zeta)=0$. We may write $\zeta$ as 
$\zeta = \xi^2$, where ${\rm Re}(\xi)>0$. By \eqref{pbbeauty}, 
$$
0=Q(\zeta)=  a_n\xi^n \sum_{k=0}^{\lfloor n/2 \rfloor} C_k e_{n-2k}\left(\rho_1\xi+\frac 1 {\rho_1\xi}, \ldots, \rho_n\xi+\frac 1  {\rho_n\xi}\right).
$$
 Since ${\rm Re}(\rho_j \xi + 1/(\rho_j \xi))>0$ for all $1\leq j \leq n$, the Grace--Walsh--Szeg\H{o} Theorem provides a number $\eta \in \mathbb{C}$, with ${\rm Re}(\eta)>0$, such that 
$$
0=\sum_{k=0}^{\lfloor n/2 \rfloor} C_k e_{n-2k}\left(\eta, \ldots, \eta\right)= \sum_{k=0}^{\lfloor n/2 \rfloor} C_k\binom n {2k} \eta^{n-2k}=: \eta^n q_n\left(\frac 1 {\eta^2}\right).  
$$
Since ${\rm Re}(\eta)>0$, we have $1/\eta^2 \in \mathbb{C} \setminus \{ x \in \mathbb{R}: x\leq  0\}$. Hence, the desired contradiction follows if we can prove that all the zeros of the polynomial $q_n(x)$ are real and negative. This follows from the identity 
\begin{eqnarray*}
\sum_{k=0}^{\lfloor n/2 \rfloor} C_k \binom n {2k}x^{k}(1+x)^{n-2k}&=&
\sum_{k=0}^n \frac 1 {n+1} \binom {n+1} k \binom {n+1} {k+1} x^k\\ 
&=& \frac 1 {n+1} (1-x)^nP_n^{(1,1)}\left(\frac {1+x}{1-x} \right), 
\end{eqnarray*}
where $\{P_n^{(1,1)}(x)\}_n$ are \emph{Jacobi polynomials}, see \cite[p.~254]{pbRai}.  The zeros of the Jacobi polynomials $\{P_n^{(1,1)}(x)\}_n$ are located in the interval $(-1,1)$. Note that the first identity in the equation above follows 
immediately from \eqref{pbbeauty}.
 \end{proof}

\section{Historical notes}
Here are some complementary historical notes about the origin of some of the central notions of this chapter. 

Although some combinatorial polynomials such as the Eulerian polynomials have been known  to be $\gamma$-positive for at least 45 years \cite{pbFoaSc},  Gal \cite{pbGal} and the author \cite{pbSign} realized the relevance of $\gamma$-positivity to topological combinatorics and in particular to the Charney--Davis conjecture.

Multivariate stable polynomials and similar classes of polynomials have been studied in many different areas. For their importance in control theory, see \cite{pbFettBas} and the references therein. In statistical mechanics they play an important role in Lee and Yang's approach to the study of phase transitions \cite{pbLY2, pbLY1}. In PDE theory so called hyperbolic polynomials play an important part in the existence of a fundamental solution to a linear PDE with constant coefficients, see \cite{pbHor}. The importance of stable polynomials in matroid theory was first realized in \cite{pbCOSW}. An important application of stable polynomials to a problem in combinatorics is GurvitsÕ proof of 
a vast generalization of the Van der Waerden conjecture, \cite{pbGurvits}. A recent application is the spectacular solution to the Kadison--Singer problem by Marcus \emph{et al.} \cite{pbKad}. See the surveys \cite{pbPemsur, pbWaS} for further applications of stable polynomials. 

The notion of HPP and WHPP matroids were introduced in Choe \emph{et al.} \cite{pbCOSW}. The strong Rayleigh property was introduced for matroids by Choe and Wagner \cite{pbCW}, and extended to discrete probability measures and studied extensively in \cite{pbJAMS}. 

\bibliography{surbib}{}

\begin{thebibliography}{100}

\bibitem{pbASW}
Michael Aissen, I.~J. Schoenberg, and A.~M. Whitney.
\newblock On the generating functions of totally positive sequences. {I}.
\newblock {\em J. Analyse Math.}, 2:93--103, 1952.

\bibitem{pbAth05}
Christos~A. Athanasiadis.
\newblock Ehrhart polynomials, simplicial polytopes, magic squares and a
  conjecture of {S}tanley.
\newblock {\em J. Reine Angew. Math.}, 583:163--174, 2005.

\bibitem{pbAthPac}
Christos~A. Athanasiadis.
\newblock Flag subdivisions and {$\gamma$}-vectors.
\newblock {\em Pacific J. Math.}, 259(2):257--278, 2012.

\bibitem{pbBend}
Edward~A. Bender.
\newblock Central and local limit theorems applied to asymptotic enumeration.
\newblock {\em J. Combinatorial Theory Ser. A}, 15:91--111, 1973.

\bibitem{pbBen}
Moussa Benoumhani.
\newblock On the modes of the independence polynomial of the centipede.
\newblock {\em J. Integer Seq.}, 15(5):Article 12.5.1, 12, 2012.

\bibitem{pbBjBr}
Anders Bj{\"o}rner and Francesco Brenti.
\newblock {\em Combinatorics of {C}oxeter groups}, volume 231 of {\em Graduate
  Texts in Mathematics}.
\newblock Springer, New York, 2005.

\bibitem{pbBon1}
Mikl{\'o}s B{\'o}na.
\newblock Symmetry and unimodality in {$t$}-stack sortable permutations.
\newblock {\em J. Combin. Theory Ser. A}, 98(1):201--209, 2002.

\bibitem{pbBonaGS}
Mikl{\'o}s B{\'o}na.
\newblock Real zeros and normal distribution for statistics on {S}tirling
  permutations defined by {G}essel and {S}tanley.
\newblock {\em SIAM J. Discrete Math.}, 23(1):401--406, 2008/09.

\bibitem{pbInv}
Julius Borcea and Petter Br{\"a}nd{\'e}n.
\newblock The {L}ee-{Y}ang and {P}\'olya-{S}chur programs. {I}. {L}inear
  operators preserving stability.
\newblock {\em Invent. Math.}, 177(3):541--569, 2009.

\bibitem{pbAnn}
Julius Borcea and Petter Br{\"a}nd{\'e}n.
\newblock P\'olya-{S}chur master theorems for circular domains and their
  boundaries.
\newblock {\em Ann. of Math. (2)}, 170(1):465--492, 2009.

\bibitem{pbLondon}
Julius. Borcea and Petter. Br{\"a}nd{\'e}n.
\newblock Multivariate {P}\'olya-{S}chur classification problems in the {W}eyl
  algebra.
\newblock {\em Proc. Lond. Math. Soc. (3)}, 101(1):73--104, 2010.

\bibitem{pbJAMS}
Julius Borcea, Petter Br{\"a}nd{\'e}n, and Thomas~M. Liggett.
\newblock Negative dependence and the geometry of polynomials.
\newblock {\em J. Amer. Math. Soc.}, 22(2):521--567, 2009.

\bibitem{pbBoMo}
George Boros and Victor Moll.
\newblock {\em Irresistible integrals}.
\newblock Cambridge University Press, Cambridge, 2004.
\newblock Symbolics, analysis and experiments in the evaluation of integrals.

\bibitem{pbCounter}
Petter Br{\"a}nd{\'e}n.
\newblock Counterexamples to the {N}eggers-{S}tanley conjecture.
\newblock {\em Electron. Res. Announc. Amer. Math. Soc.}, 10:155--158
  (electronic), 2004.

\bibitem{pbSign}
Petter Br{\"a}nd{\'e}n.
\newblock Sign-graded posets, unimodality of {$W$}-polynomials and the
  {C}harney-{D}avis conjecture.
\newblock {\em Electron. J. Combin.}, 11(2):Research Paper 9, 15 pp.
  (electronic), 2004/06.

\bibitem{pbTrans}
Petter Br{\"a}nd{\'e}n.
\newblock On linear transformations preserving the {P}\'olya frequency
  property.
\newblock {\em Trans. Amer. Math. Soc.}, 358(8):3697--3716 (electronic), 2006.

\bibitem{pbHPP}
Petter Br{\"a}nd{\'e}n.
\newblock Polynomials with the half-plane property and matroid theory.
\newblock {\em Adv. Math.}, 216(1):302--320, 2007.

\bibitem{pbAct}
Petter Br{\"a}nd{\'e}n.
\newblock Actions on permutations and unimodality of descent polynomials.
\newblock {\em European J. Combin.}, 29(2):514--531, 2008.

\bibitem{pbBrIt}
Petter Br{\"a}nd{\'e}n.
\newblock Iterated sequences and the geometry of zeros.
\newblock {\em J. Reine Angew. Math.}, 658:115--131, 2011.

\bibitem{pbBrCh}
Petter Br{\"a}nd{\'e}n and Matthew Chasse.
\newblock Infinite log--concavity for polynomial {P}\'olya frequency sequences.
\newblock Preprint, arXiv:1405.6378, 2014.

\bibitem{pbBrGo}
Petter Br{\"a}nd{\'e}n and Rafael~S. Gonz{\'a}lez~D'Le{\'o}n.
\newblock On the half-plane property and the {T}utte group of a matroid.
\newblock {\em J. Combin. Theory Ser. B}, 100(5):485--492, 2010.

\bibitem{pbMCPC}
Petter Br{\"a}nd{\'e}n, James Haglund, Mirk{\'o} Visontai, and David~G. Wagner.
\newblock Proof of the monotone column permanent conjecture.
\newblock In {\em Notions of positivity and the geometry of polynomials},
  Trends Math., pages 63--78. Birkh\"auser/Springer Basel AG, Basel, 2011.

\bibitem{pbBrLeVi}
Petter Br{\"a}nd{\'e}n, Madeleine Leander, and Mirk{\'o} Visontai.
\newblock Multivariate {E}ulerian polynomials and exclusion processes.
\newblock Preprint, arXiv:1405.6919, 2014.

\bibitem{pbBreTh}
Francesco Brenti.
\newblock Unimodal, log-concave and {P}\'olya frequency sequences in
  combinatorics.
\newblock {\em Mem. Amer. Math. Soc.}, 81(413):viii+106, 1989.

\bibitem{pbBreS}
Francesco Brenti.
\newblock Log-concave and unimodal sequences in algebra, combinatorics, and
  geometry: an update.
\newblock In {\em Jerusalem combinatorics '93}, volume 178 of {\em Contemp.
  Math.}, pages 71--89. Amer. Math. Soc., Providence, RI, 1994.

\bibitem{pbBreEu}
Francesco Brenti.
\newblock {$q$}-{E}ulerian polynomials arising from {C}oxeter groups.
\newblock {\em European J. Combin.}, 15(5):417--441, 1994.

\bibitem{pbBrWe}
Francesco Brenti and Volkmar Welker.
\newblock {$f$}-vectors of barycentric subdivisions.
\newblock {\em Math. Z.}, 259(4):849--865, 2008.

\bibitem{pbBrRo}
Winfried Bruns and Tim R{\"o}mer.
\newblock {$h$}-vectors of {G}orenstein polytopes.
\newblock {\em J. Combin. Theory Ser. A}, 114(1):65--76, 2007.

\bibitem{pbBry}
Thomas Brylawski.
\newblock Constructions.
\newblock In {\em Theory of matroids}, volume~26 of {\em Encyclopedia Math.
  Appl.}, pages 127--223. Cambridge Univ. Press, Cambridge, 1986.

\bibitem{pbChDa}
Ruth Charney and Michael Davis.
\newblock The {E}uler characteristic of a nonpositively curved, piecewise
  {E}uclidean manifold.
\newblock {\em Pacific J. Math.}, 171(1):117--137, 1995.

\bibitem{pbChen1}
William Y.~C. Chen, Donna Q.~J. Dou, and Arthur L.~B. Yang.
\newblock Br\"and\'en's conjectures on the {B}oros-{M}oll polynomials.
\newblock {\em Int. Math. Res. Not. IMRN}, (20):4819--4828, 2013.

\bibitem{pbCHY}
William.~Y.~C. {Chen}, Robert.~X.~J. {Hao}, and Harold.~R.~L. {Yang}.
\newblock {Context-free grammars and multivariate stable polynomials over
  Stirling permutations}.
\newblock {\em ArXiv e-prints}, 2012.

\bibitem{pbCOSW}
Young-Bin Choe, James~G. Oxley, Alan~D. Sokal, and David~G. Wagner.
\newblock Homogeneous multivariate polynomials with the half-plane property.
\newblock {\em Adv. in Appl. Math.}, 32(1-2):88--187, 2004.
\newblock Special issue on the Tutte polynomial.

\bibitem{pbCW}
Youngbin Choe and David~G. Wagner.
\newblock Rayleigh matroids.
\newblock {\em Combin. Probab. Comput.}, 15(5):765--781, 2006.

\bibitem{pbChSe}
Maria Chudnovsky and Paul Seymour.
\newblock The roots of the independence polynomial of a clawfree graph.
\newblock {\em J. Combin. Theory Ser. B}, 97(3):350--357, 2007.

\bibitem{pbCoWi}
Sylvie Corteel and Lauren~K. Williams.
\newblock Tableaux combinatorics for the asymmetric exclusion process and
  {A}skey-{W}ilson polynomials.
\newblock {\em Duke Math. J.}, 159(3):385--415, 2011.

\bibitem{pbCrCsJP}
Thomas Craven and George Csordas.
\newblock Jensen polynomials and the {T}ur\'an and {L}aguerre inequalities.
\newblock {\em Pacific J. Math.}, 136(2):241--260, 1989.

\bibitem{pbCrCsIt}
Thomas Craven and George Csordas.
\newblock Iterated {L}aguerre and {T}ur\'an inequalities.
\newblock {\em JIPAM. J. Inequal. Pure Appl. Math.}, 3(3):Article 39, 14 pp.
  (electronic), 2002.

\bibitem{pbCrCsS}
Thomas Craven and George Csordas.
\newblock Composition theorems, multiplier sequences and complex zero
  decreasing sequences.
\newblock In {\em Value distribution theory and related topics}, volume~3 of
  {\em Adv. Complex Anal. Appl.}, pages 131--166. Kluwer Acad. Publ., Boston,
  MA, 2004.

\bibitem{pbDarr}
J.~N. Darroch.
\newblock On the distribution of the number of successes in independent trials.
\newblock {\em Ann. Math. Statist.}, 35:1317--1321, 1964.

\bibitem{pbDeps}
Emanuele Delucchi, Aaron Pixton, and Lucas Sabalka.
\newblock Face vectors of subdivided simplicial complexes.
\newblock {\em Discrete Math.}, 312(2):248--257, 2012.

\bibitem{pbEdrei}
Albert Edrei.
\newblock On the generating functions of totally positive sequences. {II}.
\newblock {\em J. Analyse Math.}, 2:104--109, 1952.

\bibitem{pbEhr1}
E.~Ehrhart.
\newblock Sur un probl\`eme de g\'eom\'etrie diophantienne lin\'eaire. {I}.
  {P}oly\`edres et r\'eseaux.
\newblock {\em J. Reine Angew. Math.}, 226:1--29, 1967.

\bibitem{pbEhr2}
E.~Ehrhart.
\newblock Sur un probl\`eme de g\'eom\'etrie diophantienne lin\'eaire. {II}.
  {S}yst\`emes diophantiens lin\'eaires.
\newblock {\em J. Reine Angew. Math.}, 227:25--49, 1967.

\bibitem{pbFettBas}
Alfred Fettweis and Sankar Basu.
\newblock New results on stable multidimensional polynomials. {I}. {C}ontinuous
  case.
\newblock {\em IEEE Trans. Circuits and Systems}, 34(10):1221--1232, 1987.

\bibitem{pbFiskB}
Steve {Fisk}.
\newblock {Polynomials, roots, and interlacing}.
\newblock {\em ArXiv e-prints}, 0612833, 2006.

\bibitem{pbFisk1}
Steve {Fisk}.
\newblock {Questions about determinants and polynomials}.
\newblock {\em ArXiv e-prints}, 0808.1850, 2008.

\bibitem{pbFoaSc}
Dominique Foata and Marcel-P. Sch{\"u}tzenberger.
\newblock {\em Th\'eorie g\'eom\'etrique des polyn\^omes eul\'eriens}.
\newblock Lecture Notes in Mathematics, Vol. 138. Springer-Verlag, Berlin-New
  York, 1970.

\bibitem{pbFoSt}
Dominique Foata and Volker Strehl.
\newblock Euler numbers and variations of permutations.
\newblock In {\em Colloquio {I}nternazionale sulle {T}eorie {C}ombinatorie
  ({R}oma, 1973), {T}omo {I}}, pages 119--131. Atti dei Convegni Lincei, No.
  17. Accad. Naz. Lincei, Rome, 1976.

\bibitem{pbFrob}
G.~{Frobenius}.
\newblock {\"Uber die Bernoullischen Zahlen und die Eulerschen Polynome.}
\newblock {\em {Berl. Ber.}}, 1910:809--847, 1910.

\bibitem{pbGal}
{\'S}wiatos{\l}aw~R. Gal.
\newblock Real root conjecture fails for five- and higher-dimensional spheres.
\newblock {\em Discrete Comput. Geom.}, 34(2):269--284, 2005.

\bibitem{pbGash}
Vesselin Gasharov.
\newblock Incomparability graphs of {$(3+1)$}-free posets are {$s$}-positive.
\newblock In {\em Proceedings of the 6th {C}onference on {F}ormal {P}ower
  {S}eries and {A}lgebraic {C}ombinatorics ({N}ew {B}runswick, {NJ}, 1994)},
  volume 157, pages 193--197, 1996.

\bibitem{pbGesselPC}
Ira~M. Gessel.
\newblock Personal communication, 2005.

\bibitem{pbGurvits}
Leonid Gurvits.
\newblock Van der {W}aerden/{S}chrijver-{V}aliant like conjectures and stable
  (aka hyperbolic) homogeneous polynomials: one theorem for all.
\newblock {\em Electron. J. Combin.}, 15(1):Research Paper 66, 26, 2008.
\newblock With a corrigendum.

\bibitem{pbHOW}
James Haglund, Ken Ono, and David~G. Wagner.
\newblock Theorems and conjectures involving rook polynomials with only real
  zeros.
\newblock In {\em Topics in number theory ({U}niversity {P}ark, {PA}, 1997)},
  volume 467 of {\em Math. Appl.}, pages 207--221. Kluwer Acad. Publ.,
  Dordrecht, 1999.

\bibitem{pbHaVi}
James Haglund and Mirk{\'o} Visontai.
\newblock Stable multivariate {E}ulerian polynomials and generalized {S}tirling
  permutations.
\newblock {\em European J. Combin.}, 33(4):477--487, 2012.

\bibitem{pbHam}
Yahya~Ould Hamidoune.
\newblock On the numbers of independent {$k$}-sets in a claw free graph.
\newblock {\em J. Combin. Theory Ser. B}, 50(2):241--244, 1990.

\bibitem{pbHV}
J.~William Helton and Victor Vinnikov.
\newblock Linear matrix inequality representation of sets.
\newblock {\em Comm. Pure Appl. Math.}, 60(5):654--674, 2007.

\bibitem{pbHibibook}
Takayuki {Hibi}.
\newblock {\em {Algebraic combinatorics on convex polytopes.}}
\newblock Glebe: Carslaw Publications, 1992.

\bibitem{pbHor}
Lars H{\"o}rmander.
\newblock {\em The analysis of linear partial differential operators. {II}}.
\newblock Classics in Mathematics. Springer-Verlag, Berlin, 2005.
\newblock Differential operators with constant coefficients, Reprint of the
  1983 original.

\bibitem{pbHuh}
June Huh.
\newblock Milnor numbers of projective hypersurfaces and the chromatic
  polynomial of graphs.
\newblock {\em J. Amer. Math. Soc.}, 25(3):907--927, 2012.

\bibitem{pbHuhKatz}
June Huh and Eric Katz.
\newblock Log-concavity of characteristic polynomials and the {B}ergman fan of
  matroids.
\newblock {\em Math. Ann.}, 354(3):1103--1116, 2012.

\bibitem{pbLass}
Bodo Lass.
\newblock Mehler formulae for matching polynomials of graphs and independence
  polynomials of clawfree graphs.
\newblock {\em J. Combin. Theory Ser. B}, 102(2):411--423, 2012.

\bibitem{pbLY2}
T.~D. Lee and C.~N. Yang.
\newblock Statistical theory of equations of state and phase transitions. {II}.
  {L}attice gas and {I}sing model.
\newblock {\em Physical Rev. (2)}, 87:410--419, 1952.

\bibitem{pbLenz}
Matthias Lenz.
\newblock The {$f$}-vector of a representable-matroid complex is log-concave.
\newblock {\em Adv. in Appl. Math.}, 51(5):543--545, 2013.

\bibitem{pbLyons}
Russell Lyons.
\newblock Determinantal probability measures.
\newblock {\em Publ. Math. Inst. Hautes \'Etudes Sci.}, (98):167--212, 2003.

\bibitem{pbKad}
Adam Marcus, Daniel~A Spielman, and Nikhil Srivastava.
\newblock Interlacing families ii: Mixed characteristic polynomials and the
  kadison-singer problem.
\newblock Preprint, arXiv:1306.3969, 2013.

\bibitem{pbMcSa}
Peter R.~W. McNamara and Bruce~E. Sagan.
\newblock Infinite log-concavity: developments and conjectures.
\newblock {\em Adv. in Appl. Math.}, 44(1):1--15, 2010.

\bibitem{pbMustPay}
Mircea Musta{\c{t}}{\v{a}} and Sam Payne.
\newblock Ehrhart polynomials and stringy {B}etti numbers.
\newblock {\em Math. Ann.}, 333(4):787--795, 2005.

\bibitem{pbNeg}
Joseph Neggers.
\newblock Representations of finite partially ordered sets.
\newblock {\em J. Combin. Inform. System Sci.}, 3(3):113--133, 1978.

\bibitem{pbNPT}
Eran Nevo, T.~Kyle Petersen, and Bridget~Eileen Tenner.
\newblock The {$\gamma$}-vector of a barycentric subdivision.
\newblock {\em J. Combin. Theory Ser. A}, 118(4):1364--1380, 2011.

\bibitem{pbObr}
N.~{Obreschkoff}.
\newblock {Verteilung und Berechnung der Nullstellen reeller Polynome.}
\newblock {Hochschulb\"ucher f\"ur Mathematik. 55. Berlin: VEB Deutscher Verlag
  der Wissenschaften. VIII, 298 S. mit 2 Abb. (1963).}, 1963.

\bibitem{pbOhHi}
Hidefumi Ohsugi and Takayuki Hibi.
\newblock Special simplices and {G}orenstein toric rings.
\newblock {\em J. Combin. Theory Ser. A}, 113(4):718--725, 2006.

\bibitem{pbPayne}
Sam Payne.
\newblock Ehrhart series and lattice triangulations.
\newblock {\em Discrete Comput. Geom.}, 40(3):365--376, 2008.

\bibitem{pbPemSu2}
Robin Pemantle.
\newblock Towards a theory of negative dependence.
\newblock {\em J. Math. Phys.}, 41(3):1371--1390, 2000.
\newblock Probabilistic techniques in equilibrium and nonequilibrium
  statistical physics.

\bibitem{pbPemsur}
Robin Pemantle.
\newblock Hyperbolicity and stable polynomials in combinatorics and
  probability.
\newblock In {\em Current developments in mathematics, 2011}, pages 57--123.
  Int. Press, Somerville, MA, 2012.

\bibitem{pbPePe}
Robin {Pemantle} and Yuval {Peres}.
\newblock {Concentration of Lipschitz functionals of determinantal and other
  strong Rayleigh measures.}
\newblock {\em {Comb. Probab. Comput.}}, 23(1):140--160, 2014.

\bibitem{pbPete}
T.~Kyle Petersen.
\newblock Two-sided {E}ulerian numbers via balls in boxes.
\newblock {\em Math. Mag.}, 86(3):159--176, 2013.

\bibitem{pbPitman}
Jim Pitman.
\newblock Probabilistic bounds on the coefficients of polynomials with only
  real zeros.
\newblock {\em J. Combin. Theory Ser. A}, 77(2):279--303, 1997.

\bibitem{pbPolyaSchur}
G.~{P\'olya} and I.~{Schur}.
\newblock {\"Uber zwei Arten von Faktorenfolgen in der Theorie der
  algebraischen Gleichungen.}
\newblock {\em {J. Reine Angew. Math.}}, 144:89--113, 1914.

\bibitem{pbPRW}
Alex Postnikov, Victor Reiner, and Lauren Williams.
\newblock Faces of generalized permutohedra.
\newblock {\em Doc. Math.}, 13:207--273, 2008.

\bibitem{pbRS}
Q.~I. Rahman and G.~Schmeisser.
\newblock {\em Analytic theory of polynomials}, volume~26 of {\em London
  Mathematical Society Monographs. New Series}.
\newblock The Clarendon Press Oxford University Press, Oxford, 2002.

\bibitem{pbRai}
Earl~D. Rainville.
\newblock {\em Special functions}.
\newblock The Macmillan Co., New York, 1960.

\bibitem{pbReWe}
Victor Reiner and Volkmar Welker.
\newblock On the {C}harney-{D}avis and {N}eggers-{S}tanley conjectures.
\newblock {\em J. Combin. Theory Ser. A}, 109(2):247--280, 2005.

\bibitem{pbSaSc}
Carla~D. Savage and Michael~J. Schuster.
\newblock Ehrhart series of lecture hall polytopes and {E}ulerian polynomials
  for inversion sequences.
\newblock {\em J. Combin. Theory Ser. A}, 119(4):850--870, 2012.

\bibitem{pbSaVi}
Carla.~D. Savage and Mirk\'o Visontai.
\newblock The $\mathbf{s}$-{E}ulerian polynomials have only real roots.
\newblock {\em Trans. Amer. Math. Soc.}, to appear.

\bibitem{pbSWG}
Louis~W. Shapiro, Wen~Jin Woan, and Seyoum Getu.
\newblock Runs, slides and moments.
\newblock {\em SIAM J. Algebraic Discrete Methods}, 4(4):459--466, 1983.

\bibitem{pbShaWa}
John Shareshian and Michelle~L. Wachs.
\newblock Eulerian quasisymmetric functions.
\newblock {\em Adv. Math.}, 225(6):2921--2966, 2010.

\bibitem{pbStandec}
Richard~P. Stanley.
\newblock Decompositions of rational convex polytopes.
\newblock {\em Ann. Discrete Math.}, 6:333--342, 1980.
\newblock Combinatorial mathematics, optimal designs and their applications
  (Proc. Sympos. Combin. Math. and Optimal Design, Colorado State Univ., Fort
  Collins, Colo., 1978).

\bibitem{pbStanS}
Richard~P. Stanley.
\newblock Log-concave and unimodal sequences in algebra, combinatorics, and
  geometry.
\newblock In {\em Graph theory and its applications: {E}ast and {W}est
  ({J}inan, 1986)}, volume 576 of {\em Ann. New York Acad. Sci.}, pages
  500--535. New York Acad. Sci., New York, 1989.

\bibitem{pbStanComCom}
Richard~P. Stanley.
\newblock {\em Combinatorics and commutative algebra}, volume~41 of {\em
  Progress in Mathematics}.
\newblock Birkh\"auser Boston, Inc., Boston, MA, second edition, 1996.

\bibitem{pbStanChr}
Richard~P. Stanley.
\newblock Graph colorings and related symmetric functions: ideas and
  applications: a description of results, interesting applications, \& notable
  open problems.
\newblock {\em Discrete Math.}, 193(1-3):267--286, 1998.
\newblock Selected papers in honor of Adriano Garsia (Taormina, 1994).

\bibitem{pbStanEC2}
Richard~P. Stanley.
\newblock {\em Enumerative combinatorics. {V}ol. 2}, volume~62 of {\em
  Cambridge Studies in Advanced Mathematics}.
\newblock Cambridge University Press, Cambridge, 1999.
\newblock With a foreword by Gian-Carlo Rota and appendix 1 by Sergey Fomin.

\bibitem{pbStanEn}
Richard~P. Stanley.
\newblock {\em Enumerative combinatorics. {V}olume 1}, volume~49 of {\em
  Cambridge Studies in Advanced Mathematics}.
\newblock Cambridge University Press, Cambridge, second edition, 2012.

\bibitem{pbStanPC}
Richard~P. Stanley.
\newblock Personal communication, May 2008.

\bibitem{pbStein}
Einar Steingr{\'{\i}}msson.
\newblock Permutation statistics of indexed permutations.
\newblock {\em European J. Combin.}, 15(2):187--205, 1994.

\bibitem{pbStem1}
John~R. Stembridge.
\newblock Counterexamples to the poset conjectures of {N}eggers, {S}tanley, and
  {S}tembridge.
\newblock {\em Trans. Amer. Math. Soc.}, 359(3):1115--1128 (electronic), 2007.

\bibitem{pbStem2}
John~R. Stembridge.
\newblock Coxeter cones and their {$h$}-vectors.
\newblock {\em Adv. Math.}, 217(5):1935--1961, 2008.

\bibitem{pbTutte}
W.~T. Tutte.
\newblock {\em Graph theory}, volume~21 of {\em Encyclopedia of Mathematics and
  its Applications}.
\newblock Addison-Wesley Publishing Company, Advanced Book Program, Reading,
  MA, 1984.
\newblock With a foreword by C. St. J. A. Nash-Williams.

\bibitem{pbVinS}
Victor Vinnikov.
\newblock L{MI} representations of convex semialgebraic sets and determinantal
  representations of algebraic hypersurfaces: past, present, and future.
\newblock In {\em Mathematical methods in systems, optimization, and control},
  volume 222 of {\em Oper. Theory Adv. Appl.}, pages 325--349.
  Birkh\"auser/Springer Basel AG, Basel, 2012.

\bibitem{pbViWi}
Mirk{\'o} Visontai and Nathan Williams.
\newblock Stable multivariate {$W$}-{E}ulerian polynomials.
\newblock {\em J. Combin. Theory Ser. A}, 120(7):1929--1945, 2013.

\bibitem{pbWa2}
David~G. Wagner.
\newblock Enumeration of functions from posets to chains.
\newblock {\em European J. Combin.}, 13(4):313--324, 1992.

\bibitem{pbWa1}
David~G. Wagner.
\newblock Total positivity of {H}adamard products.
\newblock {\em J. Math. Anal. Appl.}, 163(2):459--483, 1992.

\bibitem{pbWaS}
David~G. Wagner.
\newblock Multivariate stable polynomials: theory and applications.
\newblock {\em Bull. Amer. Math. Soc. (N.S.)}, 48(1):53--84, 2011.

\bibitem{pbWaWe}
David~G. Wagner and Yehua Wei.
\newblock A criterion for the half-plane property.
\newblock {\em Discrete Math.}, 309(6):1385--1390, 2009.

\bibitem{pbLY1}
C.~N. Yang and T.~D. Lee.
\newblock Statistical theory of equations of state and phase transitions. {I}.
  {T}heory of condensation.
\newblock {\em Physical Rev. (2)}, 87:404--409, 1952.

\end{thebibliography}
\bibliographystyle{plain}

\end{document}